\documentclass{amsart}

\usepackage[T1]{fontenc}
\usepackage{times}
\usepackage{amssymb,epsfig,verbatim}
\usepackage{calligra,mathrsfs}
\theoremstyle{plain}
\usepackage[T1]{fontenc}
\usepackage[utf8]{inputenc}
\usepackage[english]{babel}
\usepackage{amssymb}
\usepackage{amsthm}
\usepackage{graphics}
\usepackage{amsmath}
\usepackage{amstext}
\usepackage{multirow}
\usepackage{enumerate}
\usepackage{hyperref}
\usepackage{graphicx}
\usepackage{amsmath}
\usepackage{amssymb}
\usepackage{amsthm}
\usepackage{amstext}
\usepackage{epstopdf}
\usepackage{mathrsfs}
\usepackage{amscd}
\usepackage{tikz}
\usetikzlibrary{cd}
\AtBeginDocument{%
   \def\MR#1{}
}

\newtheorem{thm}{Theorem}[section]
\newtheorem{dfn}[thm]{Definition}
\newtheorem{lemma}[thm]{Lemma}
\newtheorem{prop}[thm]{Proposition}

\newtheorem{cor}[thm]{Corollary}
\newtheorem{question}[thm]{Question}

\newtheorem{THM}{Theorem}

\newtheorem{COR}[THM]{Corollary}
\newtheorem{ex}[thm]{Example}
\theoremstyle{remark}
\newtheorem{remark}[thm]{Remark}
\newcommand{\mb}{\mathbb}

\newcommand{\mc}{\mathcal}

\newcommand{\C}{\mb C}

\newcommand{\F}{\mc F}
\newcommand{\G}{\mc G}

\newcommand\restr[2]{{
  \left.\kern-\nulldelimiterspace 
  #1 
  \vphantom{\big|} 
  \right|_{#2} 
  }}

\DeclareMathOperator{\codim}{codim}
\DeclareMathOperator{\sing}{sing}

\DeclareMathOperator{\Pic}{Pic}
\DeclareMathOperator{\Var}{Var}
\DeclareMathOperator{\Div}{Div}

\DeclareMathOperator{\CS}{CS}
\DeclareMathOperator{\Aff}{Aff}
\DeclareMathOperator{\SL}{SL}

\DeclareMathOperator{\Res}{Res}

\DeclareMathOperator{\Alb}{Alb}
\DeclareMathOperator{\alb}{alb}

\DeclareMathOperator{\ord}{ord}
\DeclareMathOperator{\Hom}{Hom}

\DeclareMathOperator{\Irr}{Irr}
\DeclareMathOperator{\ddiv}{div}
\DeclareMathOperator{\trdeg}{tr\, deg_{\mathbb C}}
\DeclareMathOperator{\HomSheaf}{\mathscr{H}\text{\kern -3pt {\calligra\large om}}\,}

\DeclareMathOperator{\Base}{\mathcal{B}}
\DeclareMathOperator{\II}{\mathrm{II}}

\newcommand{\ie}{{\it{i.e.}}}

\DeclareMathOperator{\End}{End}

\numberwithin{equation}{section}

\addtocounter{section}{0}             
\numberwithin{equation}{section}       

\title[Closed meromorphic $1$-forms]{Closed meromorphic $1$-forms}
\author{Jorge Vit\'orio Pereira}


\date{\today}

\begin{document}

\begin{abstract}
   We review properties of closed meromorphic $1$-forms and of the foliations defined
   by them. We  present and explain classical results from foliation theory, like index theorems, existence of separatrices, and resolution of singularities under the lenses of the theory of closed meromorphic $1$-forms and flat meromorphic connections.
   We apply the theory to investigate the algebraicity of separatrices in a semi-global setting (neighborhood of a compact curve
   contained in the singular set of the foliation), and the geometry of smooth hypersurfaces with numerically trivial normal
   bundle on compact Kähler manifolds.
\end{abstract}

\maketitle
\setcounter{tocdepth}{1}
\tableofcontents

\section{Introduction}

In this text, we investigate and review properties of closed meromorphic $1$-forms and of the foliations defined
by them, with special attention to the case of compact Kähler ambient spaces. We have three goals. Our first goal is mainly expository and consists in presenting basic properties of closed meromorphic $1$-forms, and using them to motivate and explain well-known foundational results of the theory of codimension one singular holomorphic foliations like indexes theorems, the existence of separatrices, and resolution of singularities. The second goal is to review old, and present new results about  smooth hypersurfaces with numerically trivial normal bundles on compact Kähler manifolds which are proved by exploring the geometry of foliations defined by closed meromorphic $1$-forms. The third goal, inspired by the second, is to investigate in a semi-global setting, \ie in a neighborhood of a compact curve, separatrices for codimension one foliations on projective threefolds.

The three goals are not independent but instead intertwined. The foundational results phrased in terms of closed meromorphic $1$-forms (and flat meromorphic connections) provide a suitable conceptual framework to discuss the existence of separatrices in a semi-global setting. Likewise, the semi-global study of separatrices for codimension one foliations on threefolds raises natural problems on the structure of neighborhoods of singular curves on surfaces and suggests analogous global problems for curves, or more generally divisors, on compact Kähler manifolds. These global problems, thanks to basic Hodge theory, can be solved with relative ease and imply new results on the geometry of the complement of smooth hypersurfaces with numerically trivial normal bundles on compact Kähler manifolds.

In the remainder of this introduction, we will describe more precisely the content of the paper.

\subsection*{Residue theorem}
We start things off in Section \ref{S:foliation} recalling the basic definitions of the foliation theory. Then, in Section \ref{S:Residue} we introduce the residues of closed meromorphic $1$-forms and establish
their main properties. The main result of the section is Weil's residue theorem for closed meromorphic $1$-forms on arbitrary complex
manifolds. The result is well-known to experts, but we could not find any
exposition of it in the recent literature. We reproduce Weil's original proof obtained as a simple
application of Stoke's theorem and  deduce from it the residue formula for flat meromorphic connections on line bundles.
As an application to foliation theory, we explain how to derive a version of Camacho-Sad index theorem for codimension one foliations
from the residue formula for flat meromorphic connections.

\subsection*{Simple singularities for closed meromorphic $1$-forms}
Section \ref{S:Simple} discusses the notion of simple singularities for closed meromorphic $1$-forms,
a concept inspired by the homonymous concept for codimension one foliations but, hopefully,  simpler to grasp.
We also explain how to transform bimeromorphically arbitrary closed meromorphic $1$-forms into closed
meromorphic $1$-forms with simple singularities. This result is a simple combination of the
embedded resolution for divisors, elimination of indeterminacies of meromorphic maps, and a procedure for
elimination of resonances of codimension one foliations recently established by Fernández Duque.
Subsection \ref{SS:irregular} describes the structure of the non-reduced part of the polar divisor of
a closed meromorphic $1$-form with simple singularities.

Section \ref{S:Simple} ends with the definition of simple singularities for codimension one foliations,
followed by a proof, due to de Almeida dos Santos, of Camacho-Sad's theorem on the existence of separatrices for
foliations on surfaces which draws inspiration from Weil's residue theorem.

\subsection*{Hodge theory}
Section \ref{S:Hodge} reviews a few basic elements of  Hodge theory on compact Kähler manifolds relevant
for the study of closed meromorphic $1$-forms. In particular, it explains how to decompose closed meromorphic $1$-forms
into sums of closed logarithmic $1$-forms (simple poles) and closed meromorphic $1$-forms of the second kind (zero residues). Section \ref{S:Hodge} also characterizes completely the residues divisors of closed meromorphic $1$-forms, as well as the polar divisors of closed
meromorphic $1$-forms of the second kind with simple singularities, when the ambient space is Kähler and compact.

\subsection*{Polar divisor of logarithmic $1$-forms}

Section \ref{S:logarithmic} is devoted to the study of logarithmic $1$-forms. It starts by explaining how the Hodge index theorem combined with the residue theorem imposes strong restrictions on the residue divisor of logarithmic $1$-forms on compact Kähler manifolds. Then, it continues
by surveying results that can be obtained through the study of certain logarithmic $1$-forms canonically attached
to real divisors with zero Chern class. These include a criterion by Totaro for the existence of fibrations, the relation of the topology of disjoint smooth hypersurfaces with proportional Chern classes, and a characterization of codimension one foliations which are pull-backs of foliations on surfaces.

\subsection*{Algebraicity criterion for semi-global separatrices}
While Sections \ref{S:foliation}--\ref{S:logarithmic} are mainly expository, the last three sections of the paper
present new results.

Section \ref{S:algebraicity} pursues the study of locally closed surfaces invariant by codimension one foliations on projective threefolds
that intersect the singular set of $\F$ on a compact subset supported on finitely many compact curves (the so-called semi-global separatrices).
Our main result on the subject consists of an algebraicity criterion for a semi-global separatrix phrased in terms of the residue divisor of Bott's connection along it.

\begin{THM}\label{THM:Logsep}
    Let $\F$ be a codimension one foliation with simple singularities on a projective manifold $X$ of dimension three.
    Let $V\subset X$ be a smooth semi-global separatrix for $\F$ and let $\nabla$ be Bott's connection
    on $V$. If the residues of $\nabla$ generate a $\mathbb Z$-submodule of $\mathbb C/\mathbb Q$ of rank at least
    two then the Zariski closure of $V$ is a $\F$-invariant algebraic surface.
\end{THM}

At first sight, the assumptions of Theorem \ref{THM:Logsep} might seen hard to check. Nevertheless, in favorable situations,
they are implied by a rather simple hypothesis that might be checked at the first non-zero jet at a point of a germ of $1$-form defining the
foliation.

\begin{COR}\label{COR:generic sing}
    Let $\F$ be a codimension one foliation on a projective manifold $X$ of dimension three.
    If there exists $p \in \sing(\F)$, with maximal ideal $\mathfrak m_p \subset \mathcal O_{X,p}$, such that $\F$ is defined by
    \[
        xyz\left( \alpha \frac{dx}{x} + \beta  \frac{dy}{y} + \gamma \frac{dz}{z} \right) \mod \, \mathfrak m_p^2 \Omega^1_{X,p} \, ,
    \]
    where $\alpha, \beta, \gamma$ are $\mathbb Z$-linearly independent complex numbers, then the Zariski closure of the local separatrices
    of $\F$ tangent to $\{x=0\}$, $\{y=0\}$, and $\{z=0\}$  are algebraic surfaces invariant by $\F$.
\end{COR}

We also have a result when the rank of the $\mathbb Z$-submodule of $\mathbb C/\mathbb Q$ generated by the (classes of) residues of Bott's connection has rank one, see Theorem \ref{T:Logsep}. Its assumptions are probably stronger than needed, as its proof is based on Ueda's theory (study of neighborhoods of smooth compact curves of zero self-intersection on surfaces) which is not presently available for singular, non-irreducible, curves.

\subsection*{Polar divisor of meromorphic $1$-forms of the second kind}
Section \ref{S:Ueda} is devoted to the study of closed meromorphic $1$-forms of the second kind. While attempting to understand Ueda's study of neighborhoods of smooth curves with zero self-intersection, we realized that his arguments admit a simple version when one considers (not necessarily smooth or reduced) curves on compact Kähler surfaces which appear as polar divisors of (multi-valued) primitives of closed meromorphic $1$-forms without residues.

\begin{THM}\label{THM:Ueda}
    Let $X$ be a compact Kähler surface and let $I \in \Div(X)$ be an effective divisor with connected support satisfying
    \[
        \mathcal O_X(I) \otimes \frac{\mathcal O_X}{\mathcal O_X(-I)} \simeq \frac{\mathcal O_X}{\mathcal O_X(-I)} \, \ie,
    \]
    the normal bundle of the (non-necessarily reduced) curve defined by $I$ is trivial.
    Then the following alternative holds:
    \begin{enumerate}
        \item there exists a  non-constant morphism  $f:X \to C$ to a curve $C$ mapping $|I|$ to a point,  or
        \item after the contraction of finitely many compact curves disjoint from the support of $I$, $X-|I|$ becomes a (perhaps singular) Stein surface.
    \end{enumerate}
\end{THM}

The proof of Theorem \ref{THM:Ueda} builds on the study of a pair of closed meromorphic $1$-forms naturally attached to the problem.
The very same pair of $1$-forms are used in our joint work with O. Thom on Grauert's formal principle for curves with trivial normal bundle on projective surfaces, reviewed in Subsection \ref{SS:formal}.

\subsection*{A remark about Stein complements}
The paper ends with Section \ref{S:Stein}, motivated by a recent result by Höring and Peternell \cite[Theorem 1.6]{https://doi.org/10.48550/arxiv.2111.03303}, which shows that if $Y$ is a smooth hypersurface on a compact Kähler manifold $X$ such that $X-Y$ is Stein then the normal bundle of $Y$ is pseudo-effective.

In view of Serre's example \cite[Chapter 6, Example 3.2]{MR0282977} of a $\mathbb P^1$-bundle over an elliptic curve admitting a section with trivial normal bundle and Stein complement, it seems natural to enquire if there are examples on higher dimensional compact Kähler manifolds of hypersurfaces with numerically trivial normal bundle and Stein complements, see  \cite[Article 2, Problems 7.1 and 7.2]{Neeman} for a similar question. In Section \ref{S:Stein}, we explore  ideas introduced in Section \ref{S:Ueda}, to establish the following result.

\begin{THM}\label{THM:Stein}
    Let $X$ be a compact Kähler manifold and $Y \subset X$ be a smooth hypersurface with numerically trivial normal bundle.
    If $X-Y$ is Stein then $\dim X=2$.
\end{THM}

We do not know if it is possible to  replace numerically trivial normal bundle by pseudo-effective normal bundle of numerical dimension zero in the  statement of Theorem \ref{THM:Stein} and still get the same conclusion.

\subsection*{Disclaimer} Due to time-space-energy limitations, the expository/survey component of this work
has many important omissions. For instance, we do not discuss the Albanese (and quasi-Albanese) varieties and maps for
compact Kähler manifolds. Other important omissions are Bogomolov's lemma \cite[Proposition 6.4]{MR3328860}; deformations of foliations defined by logarithmic $1$-forms \cite{MR1286897,MR3937325};   the topology of leaves of foliations defined by holomorphic and logarithmic $1$-forms \cite{MR1219454,MR4033932}; the study of zeros of holomorphic $1$-forms on projective manifolds \cite{MR2172952,MR3171760};
et cetera. Even for the topics discussed here, we have not tried to provide exhaustive references.

\subsection*{Acknowledgments}
This text grew out from notes for a course taught at IMPA during the Southern Hemisphere Summer of 2022
and was prepared to answer an invitation of Felipe Cano and Pepe Seade (made on June 2021)
to write a survey on a foliation theoretic topic of my choice. It is a pleasure to thank them for giving me the opportunity, and providing
the impetus, to think and share my thoughts about closed meromorphic $1$-forms and the foliations defined by them.
I also thank Andreas Höring for useful correspondence about Stein complements of hypersurfaces on compact Kähler manifolds.
I am also indebted to Maycol Falla Luza, Thiago Fassarella, Frédéric Touzet, and Sebastián Velazquez for reading parts of preliminary versions
of this work, catching quite a few misprints, and suggesting a number of improvements.  A preliminary version of this text ended up at the hands of a thorough and thoughtful referee. Their feedback saved me from a number of embarrassing mistakes and allowed me to improve the exposition.

\section{Singular holomorphic foliations}\label{S:foliation} Throughout the text we will freely use the language of foliation theory.
In this preliminary section, we collect the basic definitions of the subject. The reader familiar
with the terminology currently used in foliation theory can safely skip this section.

\subsection{Smooth foliations}
Let $X$ be a complex manifold. A smooth holomorphic foliation $\F$ of codimension $q$ on $X$ is, roughly speaking,  a
holomorphic decomposition of $X$ into a union of immersed (but not necessarily embedded) codimension $q$ submanifolds.
Formally, $\F$ is determined by the following data
\begin{enumerate}
    \item an open covering $\mathcal U$ of $X$; and
    \item for every $U \in \mathcal U$, a submersion $f_U : U \to f_U(U) \subset \mathbb C^q$ with connected fibers; and
    \item for every non-empty intersection $U\cap V$, biholomorphisms $f_{UV} : f_V(U\cap V) \to f_U(U\cap V)$ such that
    \[
        \restr{f_U}{U\cap V}  = f_{UV} \circ \restr{f_V}{U\cap V} \, .
    \]
\end{enumerate}

Consider the smallest possible equivalence relation on the points of $X$ such that $x \sim y$ when there exists $U \in \mathcal U$ such
that $f_U(x) = f_U(y)$. The equivalence classes of this relation are the leaves of $\F$. They are the immersed codimension $q$ submanifolds
alluded to in the 'rough' definition above.

The kernels of the differentials of the submersion $f_U$, patch together to form a subbundle $T_{\F} \subset T_X$ called
the tangent bundle of $\F$. The quotient $T_X/T_{\F}$ is called the normal bundle of $\F$ and its dual $N^*_{\F}$ is the
conormal bundle of $\F$. The bundle $N^*_{\F}$ is a subbundle of $\Omega^1_X$ whose sections are $1$-forms vanishing along the leaves
of $\F$.

\subsection{Singular foliations}
A codimension $q$ singular holomorphic foliation $\F$ on a complex manifold $X$ can be defined as a pair $(\F_0,\sing(\F))$ formed by
a closed analytic subset $\sing(\F)$ of $X$ and a codimension $q$ smooth foliation $\F_0$  on $X - \sing(\F)$
such that
\begin{enumerate}
    \item\label{I:codim} the codimension of $\sing(\F)$ is at least two; and
    \item\label{I:minimal} the closed subset $\sing(\F)$ is minimal in the sense that for every  closed subset $S \varsubsetneq \sing(\F)$, it does not exist a smooth foliation $\G$ of $X-S$
    such that $\restr{\G}{X-\sing(\F)} = \F_0$.
\end{enumerate}

Condition (\ref{I:codim}) is imposed to guarantee that the tangent bundle and the conormal bundle (or rather their sheaves of sections) extend through the singular set of $\F$.

The leaves of a singular foliation $\F$ are the leaves of the smooth foliation $\F_0$. A subvariety $Y\subset X$ is invariant by $\F$ if $Y$ is not contained in
$\sing(\F)$ and $Y\cap(X-\sing(\F))$ is a union of leaves of $\F_0$.

\subsection{Alternative definition}
One can also define a singular holomorphic foliation $\F$ on a complex manifold $X$  by a pair $(T_{\F},N^*_{\F})$ of
coherent subsheaves of $T_X$ and $\Omega^1_X$ such that
\begin{enumerate}
    \item the sheaf $T_{\F}$ is the annihilator of $N^*_{\F}$, \ie for every  $x \in X$
    \[
        T_{\F,x} = \{ v \in T_{X,x} \, ; \, \omega(v) = 0 \text{ for every } \omega \in N^*_{\F,x} \};
    \]
    \item the sheaf $N^*{\F}$ is the annihilator of $T_{\F}$, \ie for every $x \in X$
    \[
        N^*_{\F,x} = \{ \omega \in \Omega^1_{X,x} \, ; \, \omega(v)=0 \text{ for every } v \in T_{X,x} \};
    \]
    \item the sheaf $T_{\F}$ is involutive, \ie for every $x \in X$ and every  $v,w \in T_{\F,x}$ we have that the Lie bracket $[v,w] \in T_{\F,x}$.
\end{enumerate}
The dimension $\dim \F$ of $\F$ is the (generic) rank of $T_{\F}$ and the codimension $\codim \F$ is the (generic) rank of $N^*_{\F}$. Of course,
$\dim \F + \codim \F = \dim X$.

The singular set of $\F$ is the locus where the sheaf $T_X/T_{\F}$ (or equivalently $\Omega^1_X/N^*_{\F}$) is not
locally free. It follows from the definition that both $T_{\F}$ and $N^*_{\F}$ are reflexive subsheaves with inclusions in $T_X$ and $\Omega^1_X$
having torsion-free cokernel. In particular, the singular set of $\F$ has codimension at least two.

\subsection{Foliations defined by closed meromorphic $1$-forms}
Let $\omega$ be a closed meromorphic $1$-form on a complex manifold $X$ and let $D$ be the divisor given by $-\mathrm{div}(\omega) = (\omega)_{\infty} - (\omega)_0$.
The foliation defined by $\omega$ is the foliation with tangent sheaf given
by the kernel of the morphism
\begin{align*}
    T_{X} & \longrightarrow \mathcal O_X(D) \\
    v &\mapsto \omega(v)
\end{align*}
given by contraction with $\omega$. The involutiveness of $T_{\F}$ follows from the closedness of $\omega$. The conormal
sheaf of $\mathcal F$ is locally free (\ie, sections of a line bundle) and isomorphic to $\mathcal O_X(-D)$.

\begin{prop}
    Let $\mathcal F$ be a germ of codimension one foliation on $(\mathbb C^n,0)$ defined by
    a closed meromorphic $1$-form $\omega$. If $H$ is an irreducible component of the divisor
    of zeros and poles of $\omega$ then $H$ is invariant by $\F$.
\end{prop}
\begin{proof}
    Let $\eta$ be a  generator of $N^*_{\mathcal F}$. Since $\eta$ and $\omega$ define the same
    foliation, we can write $\eta = f \omega$ for a suitable germ of  meromorphic function $f$. Notice that
    $\ddiv(\omega) = - \ddiv(f)$.

    The closedness of $\omega$ allows us to write
    \[
        d \eta = df \wedge \omega = \frac{df}{f} \wedge f \omega = \frac{df}{f} \wedge \eta \, .
    \]
    It follows that any irreducible component of the support of $\ddiv(f)$ is invariant by $\F$.
\end{proof}

\section{Closed meromorphic 1-forms and their residues}\label{S:Residue}

\subsection{Residues and the residue divisor}\label{SS:residue}

Let $\omega$ be a closed meromorphic $1$-form on a complex manifold $X$. Let $H\subset X$ be an irreducible hypersurface.
Let $p \in H$ be a smooth point of $H$ and let $i: \overline{\mathbb D} \hookrightarrow X$ be an inclusion of a closed disc, holomorphic in the interior of $\mathbb D$, intersecting $H$ transversely at $p$.
The closedness of $\omega$ and the irreducibility of $H$ implies that the integral
\[
    \frac{1}{2\pi i} \int_{\partial \mathbb D} i^*\omega
\]
does not depend on the choice of $p$ nor the choice of $i$, as soon as both satisfy the above assumptions. By definition,
the resulting complex number is the residue of $\omega$ along $H$, denoted by $\Res_H(\omega)$. The residue divisor
of $\omega$ is the locally finite sum $\Res(\omega) = \sum_H \Res_H(\omega) H  \in \Div(X)\otimes \mathbb C$ with $H$ ranging over all
irreducible hypersurfaces of $X$.

Any complex divisor $D \in \Div(X) \otimes \mathbb C$ induces, naturally, an element of
the singular cohomology group $H^2(X,\mathbb C) = \Hom(H_2(X,\mathbb Z), \mathbb C)$. If $\sigma \in H_2(X,\mathbb Z)$ is
a homology class and $D = \sum \lambda_i H_i \in \Div(X) \otimes \mathbb C$ then the action of  $D$ in   $\sigma$
is $\sum \lambda_i H_i \cdot \sigma$ where $\cdot$ is the intersection product. The class determined by $D$ in $H^2(X,\mathbb C)$
is the (complex) Chern class of $D$ and will be denote by $c(D)$.

\begin{thm}\label{T:residue}
    Let $\omega$ be a closed meromorphic $1$-form on a (not necessarily compact) complex manifold
    $X$. Then the Chern class of $\Res(\omega)$ is equal to zero.
\end{thm}
\begin{proof}
    Let $\sigma \in H_2(X,\mathbb Z)$ be an arbitrary homology class. To prove the result it suffices
    to show that $\Res(\omega)\cdot \sigma =0$.

    Denote by $\Delta^2 \subset \mathbb R^2$ the standard $2$-simplex. Represent $\sigma$ by a singular $2$-cycle $\sum_j \sigma_j$ such that the image of each $2$-simplex $\sigma_j : \Delta^2  \to X$ either does not intersect the support of $(\omega)_{\infty}$ or intersect the support of $(\omega)_{\infty}$ transversely at a unique point that lies on the smooth locus of $|(\omega)_{\infty}|$. Assume also that the images of the boundaries of the $2$-simplexes $\sigma_j$ do not intersect $|(\omega)_{\infty}|$. Clearly, it is possible to choose a representative of $\sigma$ satisfying these assumptions.

    If the image of a simplex $\sigma_j$ intersects an irreducible component $H_i$ of $|\Res(\omega)|$ then  $\partial \sigma_j$ is homologous
    to $(H_i \cdot \sigma_j) \gamma_i$, where $\gamma_i$ is the boundary of a small (holomorphic) disk intersecting $H_i$ transversely at smooth point.
    Therefore we can write
    \[
        \Res(\omega) \cdot \sigma = \frac{1}{2\pi i} \sum_{\sigma_j \cap |\Res(\omega)| \neq \emptyset} \int_{\partial \Delta^2} \sigma_j^* \omega \, .
    \]

    Since $ \sum_j \partial \sigma_j =0$, we have that
    \[
         \sum_{ \sigma_j \cap |\Res(\omega)| \neq \emptyset} \int_{\partial \Delta^2} \sigma_j^* \omega =
        -  \sum_{\sigma_j \cap |\Res(\omega)| = \emptyset} \int_{\partial \Delta^2} \sigma_j^* \omega
    \]
    But $\int_{\partial \Delta^2} \sigma_j^*\omega = \int_{\Delta^2} d \sigma_j^* \omega=0$  whenever $\sigma_j \cap |\Res(\omega)| = \emptyset$.
    It follows that $\Res(\omega) \cdot \sigma  =0$ as claimed.
\end{proof}

\begin{remark}
    The proof above is extracted from  \cite{Weil47},  written by André Weil  during the period he was
    at Universidade de São Paulo (USP).  In the same work, he proved a converse of
    the residue theorem when the ambient is a compact Kähler manifold, see Proposition \ref{P:residueWeil} below.
\end{remark}

\subsection{Meromorphic flat connections}
If $D$ is a divisor on a complex manifold $X$ and $i \ge 0$ is an integer, we will denote by $\Omega^{i}(*D)$
the quasi-coherent sheaf of $i$-forms having polar divisor with support contained in  $|D|$, the support of $D$.
A meromorphic connection  on a line-bundle $\mathcal L$  with poles on the divisor $D$ is a
morphism of sheaves of abelian groups
\[
    \nabla : \mathcal L \to \Omega^1_X(*D) \otimes  \mathcal L
\]
which is $\mathbb C$-linear and satisfies the Leibniz rule $\nabla(f \sigma) = df \otimes \sigma  + f \nabla(\sigma)$ for any $f \in \mathcal O_X$ and any $\sigma \in \mathcal L$.

Let $\mathcal U= \{ U_i \}$ be a sufficiently fine open covering of $X$, let $\{ \sigma_i \in \mathcal L(U_i)\}$ be a
collection of nowhere vanishing sections, and let $\{ g_{ij} \in \mathcal O^*_X(U_i\cap U_j)\}$ be the cocycle determined by $\{ \sigma_i\}$, \ie $\sigma_j = g_{ij} \sigma_i$. Let $\{ \eta_i \in \Omega^1_X(*D) \}$ be the collection
of $1$-forms defined by the equalities $\nabla(\sigma_i) = \eta_i \otimes \sigma_i$.
Using the Leibniz rule to compare $\nabla(\sigma_j)$ with $\nabla(\sigma_i)$, we deduce that
\[
    \frac{dg_{ij}}{g_{ij}} = \eta_j - \eta_i \, .
\]
Reciprocally, any collection of meromorphic $1$-forms $\{\beta_i\}$ satisfying $d \log g_{ij} = \beta_j - \beta_i$ determines a
meromorphic connection on $\mathcal L$. The $1$-forms $\{ \eta_i \}$ are the (local) connection forms of $\nabla$ and depend on the
choice of trivialization. The differentials $\{d \eta_i\}$  coincide over non-empty intersections $U_i \cap U_j$ and
are independent of the choice of trivialization. The meromorphic $2$-form $\Theta \in H^0(X,\Omega^2_X(*D))$ locally defined as \[
    \frac{1}{2\pi\sqrt{-1}}d\eta_i
\]
is, by the definition, the curvature of $\nabla$. A meromorphic connection $\nabla$ is flat if its
curvature vanishes identically.

If $\nabla$ is a flat meromorphic connection,  $H\subset X$ is an irreducible hypersurface and $U_i \cap H \neq \emptyset$ then we
define the residue of $\nabla$ along $H$ as the residue of $\eta_i$ along any irreducible component of $H\cap U_i$. As in the case
of closed meromorphic $1$-forms, this definition does not depend on the choice of the open set $U_i$, or the choice of the
irreducible component of $H\cap U_i$, or the choice of the local connection form $\eta_i$. We set the residue divisor of $\nabla$ as the $\mathbb C$-divisor
\[
    \Res(\nabla) = \sum_{H} \Res_H(\nabla) \cdot H \, .
\]

\begin{thm}\label{T:residue for connections}
    Let $X$ be a complex manifold and let $\mathcal L \in \Pic(X)$ be a holomorphic line-bundle over $X$.
    If $\nabla$  is a flat meromorphic connection on $\mathcal L$ then the identity
    \[
        c(\mathcal L) = - c(\Res(\nabla))
    \]
    holds true in $H^2(X,\mathbb C)$.
\end{thm}
\begin{proof}
    Let $E(\mathcal L)$ be the total space of $\mathcal L$ and let $\pi:E(\mathcal L) \to X$ be the natural projection. Using the notation above,
    $E(\mathcal L)$ is the complex manifold obtained by identifying the points $(x, y_j) \in U_j \times \mathbb C$
    with the points $(x,g_{ij} y_i) \in U_i \times \mathbb C$ when $x \in U_i \cap U_j$.

    Let $z_i \in \mathcal L^*(U_i)$ be such that $z_i(\sigma_i)=1$. Note that $z_i$ can be interpreted as a fiberwise linear
    function on $\restr{E(\mathcal L)}{U_i}$. Since $z_i = g_{ij} z_j$ over $U_i\cap U_j$, if we set
    \[
        \omega_i = \frac{dz_i}{z_i} + \eta_i
    \]
    then $\omega_i - \omega_j = 0$. Therefore, there exists a closed meromorphic $1$-form $\omega$ on $E(\mathcal L)$
    such that $\restr{\omega}{\pi^{-1}(U_i)} = \omega_i$. As we are assuming that $\nabla$ is flat, it follows that  the $1$-form $\omega$ is closed.

    By construction, $\Res(\omega) = X + \pi^* \Res(\nabla)$ where $X \subset E(\mathcal L)$ is identified with the zero section of $\mathcal L$
    and $\pi : E(\mathcal L) \to X$ is the natural projection.  By Theorem \ref{T:residue}, $c(\Res(\omega))=0$.
    Since $\restr{c(X)}{X} = c(\mathcal L)$, it follows that
    \[
        0 = \restr{c(\Res(\omega))}{X} =  \restr{\left(c(X) + \pi^* c(\Res(\nabla)\right)}{X} = c(\mathcal L) + c(\Res(\nabla))
    \]
    in $H^2(X,\mathbb C)$ as claimed.
\end{proof}

\subsection{Residue theorem for logarithmic forms/connections (not necessarily closed/flat)}
Let $D$ be a reduced divisor on a complex manifold $X$. In \cite{MR586450}, Saito defines the sheaf of
logarithmic $1$-forms with poles on $D$ as the subsheaf $\Omega^1_X(\log D) \subset \Omega^1_X(*D)$
characterized by $\omega \in \Omega^1_X(\log D)$ if, and only if, both $\omega$ and $d\omega$ have, at worst, poles of order one.

Locally, if $D=\{ f=0\}$ and $\omega \in \Omega^1_X(\log D)$ then, according to \cite[Section 1]{MR586450}, there exist $g, h \in \mathcal O_X$ and $\eta \in \Omega^1_X$  such that
$g$ is not identically zero on any irreducible component of  $D$ and
\[
    g \omega = h \frac{df}{f} + \eta \, ,
\]

The restriction of the  quotient $h/g$ to the support of $D$  does not depend on the choices of $\eta$, $g$, and $h$. It is used by Saito to define
the residues of $\omega \in H^0(X,\Omega^1_X(\log D))$. The residue of $\omega$ along an irreducible component $H$ of $D$ is the
meromorphic function on the normalization of $H$ defined by $h/g$. In general, the residue is not a holomorphic function, it might have poles at  the pre-image of the locus where $D$ is not normal crossing. When $D$ is a normal crossing divisor, the residues of $\omega$ are holomorphic. Moreover, if $D = \sum_{i=1}^k H_i$ is simple normal crossing (\ie normal crossing with smooth irreducible components) then the sheaf $\Omega^1_X(\log D)$ fits into the exact sequence
\begin{equation}\label{E:residue exact}
    0 \to \Omega^1_X \to \Omega^1_X(\log D) \to \oplus_{i=1}^k \mathcal O_{H_i} \to 0 \, .
\end{equation}

A meromorphic connection on a line-bundle $\mathcal L$ is, by definition, a logarithmic connection if the (local) connection forms are logarithmic $1$-forms.

\begin{prop}\label{P:residueconstant}
    Let $X$ be a complex manifold, let $D$ be a reduced divisor on $X$, and let $\nabla$ be a logarithmic connection (not necessarily flat) on
    a line-bundle $\mathcal L$ over $X$. If the polar divisor $D$ is simple normal crossing and compact then the residues of $\nabla$ are constant and $c(\mathcal L) = - c(\Res(\nabla))$ in $H^2(X,\mathbb C)$.
\end{prop}
\begin{proof}
    Let $\mathcal U = \{ U_i \}$ be a sufficiently fine open covering of $X$ and let $\{ g_{ij} \} \in Z^1(\mathcal U, \mathcal O_X^*)$ be the cocycle of transition functions for $\mathcal L$.

    Let $d \mathcal A^0_X$ be the sheaf of closed complex-valued differentiable $1$-forms on $X$. The standard proof
    of de Rham's theorem relating \v{C}ech cohomology and de Rham cohomology shows that
    $
        H^1(X, d \mathcal A^0_X) \simeq H^2(X,\mathbb C) \, .
    $
    Moreover, the Chern class of $\mathcal L$ is represented by the cocycle
    \begin{equation}\label{E:Chern class}
        \left\lbrace \frac{1}{2 \pi \sqrt{-1} }  d \log g_{ij}   \right\rbrace \in  Z^1(\mathcal U, d\mathcal A^0_X) .
    \end{equation}

    Our assumptions on the polar divisor of $\nabla$ imply that the residues
    of $\nabla$ along the irreducible components of $D$ are holomorphic functions on compact and connected complex manfifolds, hence complex numbers by the maximum principle. Let $R = \sum \Res_H (\nabla) \cdot H$ be the residue divisor
    of $\nabla$. Notice that we may identify $R$ with an element of $\oplus_{i=1}^k  H^0(H_i, \mathcal O_{H_i})$.

    The connection forms $\eta_i \in \Omega_X^1(\log D)(U_i)$ of $\nabla$ determine an element
    $\eta = \{ \eta_i \} \in C^0(\mathcal U, \Omega^1_X(\log D))$ such that
    \[
        \left\lbrace \frac{1}{2 \pi \sqrt{-1}} ( \eta_j - \eta_i ) \right\rbrace = \left\lbrace \frac{1}{2 \pi \sqrt{-1}} d \log g_{ij}   \right\rbrace \, .
    \]
    Notice that the cocycle $\left\lbrace \frac{1}{2 \pi \sqrt{-1}} ( \eta_j - \eta_i ) \right\rbrace  \in Z^1(\mathcal U,\Omega^1_X) \cap  Z^1(X, d\mathcal A^0_X)$ represents
    the image of $-R \in H^0(X, \oplus_{i=1}^k \mathcal O_{H_i})$ in $H^1(X,\Omega^1_X)$ under the morphism appearing in the long exact sequence deduced from the short exact sequence (\ref{E:residue exact}). Since it also represents $-c(R)$ in $H^2(X,\mathbb C)$, see
    Formula (\ref{E:Chern class}), we get  $c(\mathcal L) = - c(R)$ in $H^2(X,\mathbb C)$ as wanted.
\end{proof}

For a generalization of Proposition \ref{P:residueconstant} to (not necessarily flat) logarithmic connections on vector bundles of arbitrary rank, see \cite{MR670900}.

\subsection{Index theorem for invariant hypersurfaces}\label{SS:index}
In this subsection, we show how to derive index theorems for invariant subvarieties of codimension one foliations
using the residue theorem for flat meromorphic connections. We barely touch the surface of this important subject.
The reader interested in the topic should consult \cite{MR1649358} and references therein. Our exposition is
inspired by \cite[Chapter 3]{MR3328860}, but differs considerably from it.

Let $\F$ be a codimension one foliation on a complex manifold $X$. If $\iota : Y \hookrightarrow X$ is smooth  $\F$-invariant hypersurface  then
$\F$ induces a flat meromorphic connection on $\restr{N_{\F}^*}{Y}$ as follows. Let $\mathcal U = \{ U_i \}$ be an open
covering of a neighborhood of $H$, let $\omega_i \in N^*_{\F}(U_i)$ be local generators of $N^*_{\F}$,
and let $g_{ij} \in \mathcal O^*_X(U_i\cap U_j)$ be nowhere vanishing functions such that $\omega_i = g_{ij} \omega_j$.
Note that the line-bundle determined by the collection $\{ \iota^*g_{ij} \}$ is $\restr{N_{\F}}{Y}$.

For each open subset $U_i$, there exists a meromorphic $1$-form $\eta_i$ such that $d\omega_i = \eta_i \wedge \omega_i$.
The $1$-forms $\eta_i$ are not uniquely defined by this last identity, one may add meromorphic multiples of $\omega$ to
any given $\eta_i$  to obtain other $1$-forms with the same property. If one chooses such multiples in a way that $(\eta_i)_{\infty}$
does not contain $Y \cap U_i$ in its support, then  $\iota^*(\eta_i)$ are meromorphic $1$-forms
on $Y\cap U_i$ uniquely determined by the choice of $\omega_i$. Notice also that over non-empty intersections we have the implication:
\[
    \eta_i \wedge \omega_i  = \frac{d g_{ij}}{g_{ij}} \wedge \omega_i + \eta_j \wedge \omega_i \implies \iota^*\eta_i -  \iota^* \eta_j=  \iota^*\frac{d g_{ij}}{g_{ij}}  \quad \text{and} \quad    d ( d \omega_i ) =0 \, .
\]
Hence
\[
     d \eta_i \wedge \omega_i = 0  \implies \iota^* d \eta_i = d \iota^* \eta_i =0 \, .
\]
It follows that the collection of meromorphic $1$-forms $\{ \iota^* \eta_i \}$ defines a flat meromorphic connection $\nabla^{B}$ on the line-bundle $\restr{N^*_{\F}}{Y}$.

\begin{remark}\label{R:Bottconnection}
    The meromorphic connection $\nabla^B$ is the restriction of Bott's partial connection to the invariant hypersurface $Y$.
\end{remark}

\begin{ex}\label{E:generic sing}
    Let $\F$ be the germ of foliation on $(\mathbb C^3,0)$ defined by
    \[
        \omega = xyz \left(  \alpha \frac{dx}{x} + \beta \frac{dy}{y} + \gamma \frac{dz}{z} \right) \, , \alpha, \beta, \gamma \in \mathbb C^*,
    \]
    and let $Y = \{ x=0\}$. Observe that
    $d \omega = d \log xyz \wedge \omega$. Set $\eta_{\mu}  = d \log xyz + \mu \cdot (xyz)^{-1}\omega$. Clearly, $d\omega = \eta_{\mu} \wedge (xyz)^{-1} \omega$
    for any meromorphic function $\mu$ on $(\mathbb C^3,0)$. In order to guarantee that $(\eta_{\mu})_{\infty}$ has no poles on $Y$, we may choose $\mu = - \alpha^{-1}$. Therefore the connection $\nabla^B$ (\ie Bott's connection)  on $Y = \{ x=0\}$ is defined by the logarithmic $1$-form
    \[
         \left(1 - \frac{\beta}{\alpha} \right) \frac{d y}{y} + \left(1- \frac{ \gamma}{\alpha}\right) \frac{d z}{z} .
    \]
\end{ex}

If $S \subset Y$ is an irreducible component of the polar divisor of $\nabla^B$ then $S$ is contained in $\sing(\F) \cap Y$
and the residue of $\nabla^B$ along $S$ is the so-called variation $\Var(\F,Y,S)$ as defined by Suwa, see \cite{MR1649358}.

\begin{prop}\label{P:variation}
    Let $\F$ be a codimension one foliation defined on a complex manifold $X$. If $Y$ is a smooth  $\F$-invariant hypersurface
    then
    \[
        c(\restr{N_{\F}}{Y}) = \sum_{S} \Var(\F,Y,S) \cdot c(S) \,
    \]
    in $H^2(Y,\mathbb C)$, where the sum runs over all the irreducible components of $\sing(\F) \cap Y$.
\end{prop}
\begin{proof}
    Apply Theorem \ref{T:residue for connections} to the connection  $\nabla^B$ constructed above to get that
    $c(\restr{N^*_{\F}}{Y}) = - c(\Res(\nabla^B))$. Therefore
    \[
        c(\restr{N_{\F}}{Y}) =  c(\Res(\nabla^B)) = \sum_{S} \Var(\F,Y,S) \cdot c(S)
    \]
    according to the definition of $\Var(\F,Y,S)$.
\end{proof}

For a codimension one foliation $\F$ defined by $\omega \in H^0(X,\Omega^1_X \otimes N_{\F})$ and a smooth invariant
hypersurface $\iota : Y \hookrightarrow X$, the restriction $\restr{\omega}{Y} \in H^0(Y, \restr{\Omega^1_X}{Y}\otimes \restr{N_{\F}}{Y})$ is contained in the image of the natural morphism
\[
    N^*_Y \otimes \restr{N_{\F}}{Y} \to \restr{\Omega^1_X}{Y} \otimes \restr{N_{\F}}{Y} .
\]
Define  $Z(\F,Y)$ as the zero divisor of $\restr{\omega}{Y}$ seen as an element of $H^0(Y,N^*_Y \otimes N_{\F})$.
Observe that we can expand $Z(\F,Y)$ as $\sum_{S} Z(\F,Y,S) S$ where the sum runs over the irreducible components
of $\sing(\F)\cap Y$.

\begin{prop}\label{P:Znormal}
    Let $\F$ be a codimension one foliation defined on a complex manifold $X$. If $Y$ is a smooth invariant hypersurface
    then
    \[
         N_Y = \restr{N_{\F}}{Y}  \otimes \mathcal O_X( - Z(\F,Y)) \,
    \]
    in $\Pic(Y)$.
\end{prop}
\begin{proof}
    The invariance of $Y$ by $\F$ implies that the inclusion $\restr{N^*_{\mathcal F}}{Y}\hookrightarrow \restr{\Omega^1_X}{Y}$ factors through a morphism $\varphi:\restr{N^*_{\F}}{Y} \to  N^*_Y$. By the definition
    of $Z(\F,Y)$, the image of $\varphi$ is equal to $N^*_Y \otimes \mathcal O_X(-Z(\F,Y))$.
\end{proof}

The study of indexes for invariant subvarieties of foliations was initiated by Camacho and Sad in \cite{CamachoSad}, where they
introduced the index that now bears their name. Although they defined it only for smooth curves invariant by foliations on surfaces,
their definition works equally well for smooth hypersurfaces invariant by codimension one foliations.

Let $\omega$ be a local generator of $N^*_{\F}$ and let $f$ be a local equation for $Y$. We can write
$\omega = h df + f \eta$, where $h$ is a holomorphic function and  $\eta$ is a holomorphic $1$-form.
Observe that $\sing(\F)\cap Y= \{ f=h=0\}$. The integrability condition $\omega\wedge d\omega=0$ implies that $\iota^*(\eta/h)$ is a closed meromorphic $1$-form. The Camacho-Sad index along $S$ is defined as the residue of $\iota^*(\eta/h)$ along $S$. Since $\iota^*(\eta/h)$
is closed, its residues are complex numbers. Moreover, they do not depend on the choices involved.

As observed by Brunella in \cite[Proposition 3.1]{MR3328860}, the variation $\Var(\F,Y,S)$ is closely related to the Camacho-Sad index $\CS(\F,Y,S)$:
\begin{equation}\label{E:V=CS+Z}
    \Var(\F,Y,S) = \CS(\F,Y,S) + Z(\F,Y,S) \, .
\end{equation}

We are now ready to state, and prove, a version of the Camacho-Sad index theorem for smooth hypersurfaces
invariant by codimension foliations.

\begin{thm}\label{T:CS}
    Let $\F$ be a codimension one foliation defined on a complex manifold $X$. If $Y$ is a smooth  invariant hypersurface
    then the identity
    \[
        c(N_Y) = \sum_{S} \CS(\F,Y,S) \cdot c(S)
    \]
    holds true     in $H^2(Y,\mathbb C)$.
\end{thm}
\begin{proof}
    Combining Propositions \ref{P:variation} and \ref{P:Znormal}, we get that
    \[
        c(N_Y) = c(\Var(\F,Y)) - c(Z(\F,Y)) \, .
    \]
    The result follows from Equation (\ref{E:V=CS+Z}).
\end{proof}

\begin{remark}
    The proof of Theorem \ref{T:CS} presented above suggests that the foliation $\F$ induces a flat meromorphic connection $\nabla^{CS}$ on the
    normal bundle $Y$ such that the Camacho-Sad index of irreducible components of $\sing(\F) \cap Y$ coincide with the residues of $\nabla^{CS}$.
    This is actually the case. The line-bundle $\mathcal O_Y(Z(\F,Y))$ admits a flat logarithmic connection $\nabla^{Z}$ with  polar divisor equal to $Z(\F,Y)$ and trivial monodromy. The connection dual to the  tensor product  $(\restr{N^*_{\F}}{Y} ,\nabla^B) \otimes (\mathcal O_Y(Z(\F,Y), \nabla^Z)$ is, according to Proposition \ref{P:Znormal}, a connection on $N_Y$. It can be checked that residues of this connection coincide with the Camacho-Sad indexes of $\F$ along $Y$.

    Alternatively, we can interpret the connection $\nabla^{CS}$ geometrically by considering a deformation of $\F$ to the normal cone of $Y$ in the sense of \cite[Chapter 5]{MR1644323}.  To wit, set $Z = X \times \mathbb P^1$ and let $\F_Z$ be the codimension two foliation on $Z$ which is tangent to the natural fibration $Z \to \mathbb P^1$ and, fiberwise, coincides with $\F$. Let $M_Y$ be the blow-up of $Z$ along $Y \times \{\infty\}$ and let $\pi : M_Y \to \mathbb P^1$ be the induced fibration. The fiber $\pi^{-1}(\infty)$ has two irreducible components: one of them is a copy of $X$ and the other is the exceptional divisor $E=\mathbb P(N_Y\oplus \mathcal O_Y)$, a compactification of the total space of the normal bundle of $Y$ in $X$. Let $\tilde{\F_Z}$ be the pull-back of $\F_Z$ to $M_Y$. The $\F$-invariance of $Y$ implies that the exceptional divisor $E$ is $\tilde{\F_Z}$ invariant. The codimension one foliation induced on it coincides with the compactification of foliation on the total space  of $N_Y$ defined by $\nabla^{CS}$. As this interpretation will not be used in what follows, we leave the verification of details to the interested readers.
\end{remark}

\section{Simple singularities for closed meromorphic 1-forms}\label{S:Simple}

\subsection{Local expression for closed meromorphic 1-forms} In general, given a divisor $R \in \Div(X) \otimes \mathbb C$ with zero Chern class, it is not possible
to realize it globally as the residue divisor of a closed meromorphic $1$-form. At the same time, there are no local obstructions. To wit, if $X$ is the germ of a $n$-dimensional manifold at a point, \ie $X \simeq (\mathbb C^n,0)$, and $R= \sum_i \lambda_i H_i$
is a $\mathbb C$-divisor on $X$ written as a sum of irreducible hypersurfaces with complex coefficients then each $H_i$ is defined by
an irreducible holomorphic function $h_i$ and we can set
\[
    \eta = \sum \lambda_i \frac{dh_i}{h_i} \,
\]
as the sought closed meromorphic differential with residue divisor $R$.

\begin{lemma}\label{L:local expression}
    Let $R = \sum \lambda_i H_i$ be a divisor on $(\mathbb C^n,0)$ as above.
    Let $\omega$ be a closed meromorphic $1$-form on $(\mathbb C^n,0)$ with residue divisor equal to $R= \sum \lambda_i H_i$ (notation as above).
    Then there exists relatively prime germs of holomorphic functions $f,g$ on $(\mathbb C^n,0)$  such that
    \[
        \omega = \sum \lambda_i \frac{dh_i}{h_i} + d \left( \frac{f}{g} \right) \, .
    \]
\end{lemma}
\begin{proof}
    We follow the proof of \cite[Chapter 3, Theorem 2.1]{MR704017}.
    We can assume that $\omega$ is defined on a polydisk $U$ containing the origin of $\mathbb C^n$.
    We can also assume that the irreducible components $H_i$ of the support of $R$ are defined  on $U$ by irreducible holomorphic
    functions $h_i \in \mathcal O_{\mathbb C^n}(U_i)$. Set $\eta$ equal to $\sum \lambda_i \frac{dh_i}{h_i}$  and consider
    the closed meromorphic $1$-form $\omega - \eta$. Since $\omega - \eta$ is a closed meromorphic $1$-form with zero residues,  it follows that $\omega - \eta$ is a exact on $U-|R|$. Moreover, a local computation along the smooth locus of $H_i$ shows that any primitive $h$  of $\restr{\omega - \eta}{U-|R|}$ extends meromorphically to the complement of the singular set of $|R|$. Remmert-Stein theorem \cite[page 353, Theorem 1.3]{SCVVII} implies that $h$ extends to a meromorphic function on $U$. Since any germ of meromorphic function  is the quotient
    of two relatively prime holomorphic functions, the lemma follows.
\end{proof}

\subsection{Base locus}
Let $\omega$ be a closed meromorphic $1$-form on a complex manifold $X$.
Let $U \subset X$ be an open subset biholomorphic to a polydisk.
Over $U$, according to Lemma \ref{L:local expression}, there exists a closed logarithmic $1$-form
$\eta$ and holomorphic functions $f,g \in \mathcal O_X(U)$ without common factors such that
\[
    \restr{\omega}{U} = \eta + d \left(\frac{f}{g}\right) \, .
\]
None of the ingredients $\eta, f ,g $ of this identity are intrinsically attached
to $\omega$. Nevertheless, we claim that the ideal generated by $f$ and $g$ does not depend on the choices made above.
Indeed, if
\[
    \eta + d \left(\frac{f}{g}\right) = \tilde \eta + d \left(\frac{\tilde f}{ \tilde g}\right)
\]
then we can write
\[
    \frac{f}{g} - \frac{\tilde f}{\tilde g} = h \implies \tilde g \cdot f - \tilde f \cdot g = h g \tilde g
\]
for some some holomorphic function $h$. Since $\tilde g$ and $g$ differ multiplicatively by a unit, the
claim follows.

\begin{dfn}
    The base locus of $\omega$ is the complex analytic subspace of $X$ defined by the ideal sheaf $\Base(\omega) \subset \mathcal O_X$ locally generated by $f$ and $g$.
\end{dfn}

\subsection{Polar divisor and  irregular divisor}
If $\omega$ is a closed meromorphic $1$-form on a complex manifold $X$, we will denote by $(\omega)_{\infty}$
the polar divisor of $\omega$. We will write  $\Irr(\omega)$ for the difference of $(\omega)_{\infty}$ and the reduced divisor with the same support as $(\omega)_{\infty}$ (the irregular divisor $\omega$).

\subsection{Resonant residues}
We say that $\omega$ is resonant at a point $p$ if there exists irreducible components $H_1$, \ldots, $H_k$
of the support of $\Res(\omega)$  containing $p$, none of them contained in the support of $\Irr(\omega)$,  and strictly positive integers $n_1, \ldots, n_k$ such that
\[
    \sum_{i=1}^k n_i \Res_{H_i}(\omega) = 0
\]
If $\omega$ is not resonant for every point $p \in X$, we will say that $\omega$ is non-resonant.

\begin{lemma}\label{L:residue pull-back}
    Let $\pi:Y \to X$ be a bimeromorphic morphism and let $E$ be an irreducible component of the exceptional divisor of $\pi$.
    The residue along $E$ of $\pi^* \omega$ is equal to
    \[
        \sum_{i=1}^k \ord_E(\pi^*H_i) \cdot \Res_{H_i}(\omega) \, .
    \]
    Equivalently,  $\Res(\pi^*\omega) = \pi^*\Res(\omega)$.
\end{lemma}
\begin{proof}
    Simple local computation.
\end{proof}

Lemma \ref{L:residue pull-back} has the following  consequence. If $\omega$ is a non-resonant closed meromorphic $1$-form then every irreducible component $E$ of the exceptional divisor of a bimeromorphic morphism $\pi:Y\to X$ with center contained in the support of the residue divisor of $\omega$, but not contained in the base locus of $\omega$,  is invariant by the foliation defined by $\pi^* \omega$.

\subsection{Simple singularities of closed meromorphic $1$-forms}\label{SS:simple}
The definition below is motivated by the analogous concept for singularities of codimension one foliations, see Definition \ref{D:simple} in Subsection \ref{SS:CanoCerveau}.

\begin{dfn}
    A closed meromorphic $1$-form has simple singularities if $\Base(\omega)= \mathcal O_X$,
    $\omega$ is non-resonant, and the support of the polar divisor of $\omega$ is a simple normal crossing
    divisor.
\end{dfn}

Beware that the definition above is not a perfect analogue of the homonymous concept for codimension one foliations
since we impose no conditions on the zeros of $\omega$. The reason is that a closed $1$-form with simple singularities
is exact at a neighborhood of its zero set (see \S \ref{SS:zeros} below).

The result below appears in \cite[Proposition 8.4]{ratendo}.

\begin{prop}\label{P:trivialreduction}
    Let $\omega$ be a closed meromorphic $1$-form on a complex manifold $X$. There exists a proper bimeromorphic morphism
    $\pi : Y \to X$ such that $\pi^* \omega$ is a closed meromorphic $1$-form with simple
    singularities.
\end{prop}
\begin{proof}
    Let $\rho:Z \to X$ be the composition of the blow-up along the ideal sheaf $\Base(\omega)$, with a resolution of singularities of the resulting analytic space \cite[Theorem 2.0.1]{MR2500573} followed by an embedded resolution of the support of the polar divisor of the resulting $1$-form \cite[Theorem 2.0.2]{MR2500573}. By design, $\pi^*\omega$ is a closed meromorphic $1$-form with $\Base(f^*\omega)=\mathcal O_X$ and simple normal crossing polar divisor. The existence of a morphism  $Y \to Z$ eliminating the resonances of $f^*\omega$ follows from \cite[Theorem 2]{MR3302577}.
\end{proof}

\subsection{Local expression for simple singularities}\label{SS:disjoint}
Closed meromorphic $1$-forms with simple singularities admit  quite simple local expressions in neighborhoods
of points contained in their polar divisors.

\begin{lemma}\label{L:localsimple}
    Let $\omega$ be a germ of closed meromorphic $1$-form with simple singularities on $(\mathbb C^n,0)$.
    If $\omega$ is not holomorphic then there exists a system of coordinates $(x_1, \ldots, x_n)$ such that
    \[
        \omega = \sum_{i=1}^n \lambda_i \frac{dx_i}{x_i} + d \left( \frac{1}{x_1^{n_1} \cdots x_n^{n_n}} \right) \,
    \]
    where $\lambda_i \in \mathbb C$ and $n_i \in \mathbb Z_{\ge 0}$.
\end{lemma}
\begin{proof}
    The proof follows closely \cite[Chapter 2, proof of Proposition 2.3]{loray:hal-00016434}.
    The definition of simple singularities implies that we can write
    \[
        \omega_0 = \sum_{i=1}^n \lambda_i \frac{dx_i}{x_i} + d \left( \frac{f}{x_1^{n_1} \cdots x_n^{n_n}} \right) \,
    \]
    with $f$ equal to a germ of holomorphic function on $(\mathbb C^n,0)$ such that $f(0)\neq 0$.
    To prove the lemma, we will start from  $\omega$ as in the statement and make a change of variables
    to turn it into $\omega_0$.

    If $n_1 = \cdots = n_n = 0$ and $\lambda_j \neq0$ then we can replace $x_j$ by $x_j \cdot \exp(f/\lambda_j)$ in order to turn
    $\omega$ into $\omega_0$.

    If instead there exists $j$ such that $n_j \neq 0$ then we replace $x_j$ by $x_j \cdot u(x_1, \cdots, x_n)$ where $u$ is a unit that satisfies
    \[
         \lambda_j \cdot {x_1^{n_1} \cdots x_n^{n_n}} \cdot \log u +  u^{-n_j}=  f \, ,
    \]
    to turn $\omega$ into $\omega_0$. The existence of a unit $u$ satisfying the equation above follows from the implicit function theorem.
\end{proof}

\begin{cor}\label{C:disjoint}
    Let $\omega$ be a closed meromorphic $1$-form on a complex manifold $X$.
    If $\omega$ has simple singularities then the closure of the support of its zero set
    is disjoint from the support of the polar divisor of $\omega$.
\end{cor}
\begin{proof}
    Let $p \in X$ be an arbitrary point in the support of the polar divisor of  $\omega$.
    The germ of $\omega$ at $p$ is described by Lemma \ref{L:localsimple} and  the explicit
    formula given by it implies that the zero set of $\omega$ is empty at a neighborhood of $p$.  Since
    $p$ is arbitrary, the result follows.
\end{proof}


\subsection{Zeros}\label{SS:zeros}
As mentioned in Subsection \ref{SS:simple},  closed meromorphic $1$-forms with simple singularities
are exact at sufficiently small neighborhoods of connected components of their zero sets.

\begin{lemma}\label{L:existence primitive}
    Let $\omega$ be a closed meromorphic $1$-form with simple singularities on a complex
    manifold $X$. If  $Z$ is a connected component of the zero set of $\omega$ then there exists
    an open neighborhood $U$ of $Z$ such that $\restr{\omega}{U}$ is exact.
\end{lemma}
\begin{proof}
    According to Corollary \ref{C:disjoint} the closure of the zero set of $\omega$ does not intersect the polar
    set of $\omega$. Therefore, if $Z$ is a connected component of the zero set of $\omega$, there exists an open neighborhood
    $U$ of $Z$ such that $\restr{\omega}{U}$ is holomorphic. For every point $p \in Z$, we can choose a local primitive $f \in \mathcal O_{X,p}$
    for $\omega$ such that $f(p)=0$. If we take a neighborhood $U_p$ of $p$ such that $Z \cap U_p$ is connected then
    $\restr{\omega}{U_p} = \restr{df}{U_p}$ implies that $\restr{f}{U_p \cap Z}$ vanishes identically. Consequently, perhaps after shrinking the
    neighborhood $U$, we obtain a collection of local primitives of $\omega$ with domains covering $U$ that are uniquely determined by the
    condition $\restr{f}{U_p \cap Z}=0$. The uniqueness implies they patch together proving the exactness of   $\restr{\omega}{U}$.
\end{proof}

\subsection{The irregular divisor}\label{SS:irregular}
The next result provides obstructions to the realization of a given effective divisor as
the irregular divisor of a closed meromorphic $1$-form with simple singularities. Intriguingly,
as we will show in Subsection \ref{SS:realizairregular},  these are the only obstructions when the ambient space is compact and Kähler.

\begin{prop}\label{P:irregular}
    Let $\omega$ be a closed meromorphic $1$-form on a complex manifold $X$. If $\omega$ has simple singularities
    and  $I = \Irr(\omega)$ is the irregular divisor of $\omega$ then
    \[
        \mathcal O_X(I) \otimes \frac{\mathcal O_X}{\mathcal O_X(-I)} \simeq \frac{\mathcal O_X}{\mathcal O_X(-I)} \, .
    \]
    In other words, the normal bundle of the (perhaps non-reduced) complex analytic space defined
    by the ideal sheaf $\mathcal O_X(-I)$ is trivial.
\end{prop}
\begin{proof}
    Choose an open covering $\mathcal U=\{U_i\}$ of $X$ such that for every $U_i \in \mathcal U$ intersecting
    the support of $I$ we can write $\restr{\omega}{U_i}$ as
    \[
        \restr{\omega}{U_i} = \eta_i + d \left(\frac{1}{f_i}\right)
    \]
    where $\eta_i$ is a logarithmic $1$-form and $f_i  \in \mathcal O_X(U_i)$ defines the irregular divisor
    on $U_i$. The existence of such open covering is assured by Lemma \ref{L:localsimple}.

    Comparing these expressions on non-empty intersections $U_i \cap U_j$, we get that
    \[
        \restr{d \left(\frac{1}{f_i}\right)}{U_i\cap U_j} - \restr{d \left(\frac{1}{f_j}\right)}{U_i \cap U_j} =  \eta_{ij}
    \]
    where $\eta_{ij} \in \Omega^1_X(U_i\cap U_j)$ is a closed holomorphic $1$-form. If we set $g_{ij}$ as a suitable primitive of $\eta_{ij}$ then
    \[
        \restr{f_j}{U_i \cap U_j} = \restr{\frac{f_i}{1- g_{ij} f_i}}{U_i \cap U_j} \, .
    \]
    To conclude it suffices to observe that the cocycle
    $
        \left\{  {1-g_{ij} f_i} \right\} \in C^1(\mathcal U,\mathcal O_X^*)
    $
    defines the line-bundle $\mathcal O_X(I)$ and is trivial (actually equal to $1$) when restricted to the  complex analytic space defined
    by the ideal sheaf $\mathcal O_X(-I)$.
\end{proof}

\begin{prop}\label{P:irregularconnections}
    Let $\nabla$ be a flat  meromorphic connection on a line bundle $\mathcal L$ a complex manifold $X$.
    If the irregular divisor
    $I = \Irr(\nabla)$ of $\nabla$ has compact support then $I^2 \ge 0$. Moreover, if $I^2=0$ then
    \[
        \mathcal O_X(I) \otimes \frac{\mathcal O_X}{\mathcal O_X(-I)} \simeq \frac{\mathcal O_X}{\mathcal O_X(-I)} \, .
    \]
\end{prop}
\begin{proof}
    Choose an open covering $\mathcal U=\{U_i\}$ of $X$ such that for every $U_i \in \mathcal U$ intersecting
    the support of $I$ we can write $\restr{\nabla}{U_i} = d + \restr{\omega}{U_i}$ and
    \[
        \restr{\omega}{U_i} = \eta_i + d \left(\frac{h_i}{f_i}\right)
    \]
    where $\eta_i$ is a logarithmic $1$-form, $h_i \in \mathcal O_X(U_i)$ does not vanish on any irreducible component of the support of $\restr{I}{U_i}$, and $f_i  \in \mathcal O_X(U_i)$ defines the irregular divisor     on $U_i$.

    Arguing as in the proof of Proposition \ref{P:irregular}, we deduce that
    \[
        \restr{f_j}{U_i \cap U_j} = \restr{\frac{h_j f_i}{h_i- g_{ij} f_i}}{U_i \cap U_j} \, .
    \]
    Restricting this expression to the support of $I$, we deduce that $\mathcal O_{|I|}(I)$ is linearly equivalent to
    the effective divisor defined by $\{ \restr{ h_i}{|I|}=0\}$. Hence $I^2\ge 0$. If $I^2=0$ then the functions $h_i \in \mathcal O_X(U_i)$
    are invertible (\ie the $1$-forms $\omega_i$ are without base points) and we repeat the arguments used to prove Proposition \ref{P:irregular}
    in order to reach the same conclusion.
\end{proof}

\begin{remark}
    Proposition \ref{P:irregularconnections} applied to a flat meromorphic connection on a trivial line-bundle implies that Proposition \ref{P:irregular} also holds if one only assumes that $\omega$ is without base points.
\end{remark}

\subsection{Simple singularities for codimension one foliations}\label{SS:CanoCerveau}
For the reader's sake, we reproduce below the definition of simple singularities for codimension one foliations  due to Cano and Cerveau, as presented in \cite{MR1760842}.

\begin{dfn}\label{D:simple}
    Let $\F$ be a germ of codimension one foliation on $(\mathbb C^n,0)$. The foliation
    $\F$ has simple singularities if there exists formal coordinates $x_1,\ldots, x_n$ and an integer
    $r$, $2 \le r \le n$, (the dimension type of $\F$) such that
    $\F$ is defined by a differential form $\omega$ of one of the following types:
    \begin{enumerate}
        \item\label{I:CC1}  There are complex numbers $\lambda_i \in \mathbb C^*$ such that
        \[
            \omega = \sum_{i=1}^r \lambda_i \frac{dx_i}{x_i} \, ,
        \]
        and $\sum_{i=1}^r a_i \lambda_i = 0$ for non-negative integers $a_i$ implies $(a_1, \ldots, a_r)=0$.
        \item\label{I:CC2} There exist an integer $1\leq k \le r$,  positive integers $p_1, \ldots, p_k$ , complex numbers $\lambda_2, \ldots,\lambda_r$, and a formal power series $\psi \in t \cdot \mathbb C[[t]]$  such that
        \[
            \omega =  \sum_{i=1}^k p_i \frac{dx_i}{x_i} + \psi(x_1^{p_1}\cdots x_k^{p_k}) \sum_{i=2}^r \lambda_i \frac{dx_i}{x_i}
        \]
        and $\sum_{i=k+1}^r a_i \lambda_i = 0$ for non-negative integers $a_i$ implies $(a_{k+1}, \ldots, a_r)=0$.
    \end{enumerate}
\end{dfn}

\begin{remark}\label{R:misterious index}
    The second summation in Item (\ref{I:CC2}) of Definition \ref{D:simple} may cause some discomfort at first sight.
    However, if we allow the second summation to start at $1$, and keep all the other conditions, we get the same definition as pointed
    out in \cite[Remark 8.3]{ratendo}.
\end{remark}

Because of the results presented in this section, we may rephrase Definition \ref{D:simple} by saying that
a germ of $\F$ on $(\mathbb C^n,0)$ has simple singularities if, and only if, it is formally equivalent to a
foliation defined by a germ of closed meromorphic $1$-form with simple singularities and non-trivial polar divisor.
For a  description of simple singularities in dimension three, and their separatrices, we refer to \cite[Subsections 54 and 5.5]{MR1760842}.

The analog for codimension one foliations of Proposition \ref{P:trivialreduction} is  only  known
in dimensions two and three.

\begin{thm}\label{T:reductionofsings}
    Let $\F$ be a codimension one foliation on a complex manifold $X$ of dimension at most three.
    Then there exists a proper bimeromorphic morphism (obtained as a locally finite composition
    of blow-ups of  smooth centers) $\pi :Y \to X$ such that all the singularities of
     $\pi^*\F$  are simple.
\end{thm}

In dimension two, Theorem \ref{T:reductionofsings} is due to Seidenberg, see for instance \cite[Chapter 1]{MR3328860}. Its proof consists in
keeping blowing-up non-simple singularities (points) until getting a foliation with simple singularities (also
called reduced in dimension two). The result is not completely evident,  but the strategy to obtain the
resolution, assuming it exists, is obvious and essentially unique.
In dimension three, the problem is considerably more
delicate as one may blow-up points or curves and it is not a priori clear which one to
choose at each step of the process. Theorem \ref{T:reductionofsings} was first proved
for non-dicritical foliations by Cano and Cerveau \cite{MR1179013}, and later, for arbitrary codimension one foliations by Cano \cite{MR2144971}.
It is natural to expect that the statement holds in arbitrary dimensions, but no proof is available yet.

\subsection{The separatrix theorem} If $\F$ is a codimension one foliation on a complex manifold $X$ then
a separatrix of $\F$ through a point $p \in \sing(\F)$ is a germ of hypersurface $S$ at $p$ such that
$S-\sing(\F)$ is contained in a leaf of $\restr{\F}{X- \sing(\F)}$.

Combining resolution of singularities for foliations on surfaces
with their index theorem (Theorem \ref{T:CamachoSad}), Camacho and Sad established the following fundamental result for the study
of foliations on surfaces.

\begin{thm}\label{T:CamachoSad}
    Let $\F$ be a germ of foliation on $(\mathbb C^2,0)$. Then $\F$ admits a
    separatrix, \ie, there exists a germ of curve parameterized by  $\gamma :(\mathbb C,0) \to (\mathbb C^2,0)$
    such that $\gamma^*\omega$ for any $\omega \in N^*_{\F}$.
\end{thm}

The result for simple singularities was known at the end of the 19th century and has been traced back to the work of Briot and Bouquet.
More specifically, always assuming that $\F$ has simple singularities, if the tangent sheaf is generated by a vector field with invertible linear part then $\F$
admits exactly two separatrices. If instead the tangent sheaf is generated by a vector with non-invertible linear part
then there exists at least one separatrix tangent to the eigenspace corresponding to the non-zero eigenvalue. Importantly, the Camacho-Sad index
for this separatrix is equal to zero. See for instance \cite[Section 1]{CamachoSad} or \cite[Chapter 3, Section 2]{MR3328860}

After Camacho and Sad proved Theorem \ref{T:CamachoSad}, a number of different proofs and generalizations appeared in the literature.
The first one \cite{MR963011} was by Camacho, who extended Theorem \ref{T:CamachoSad} to singular surfaces with resolution having contractible dual graph. There are proofs by J. Cano \cite{MR1389507} of Camacho-Sad's original result, and by  M. Sebastiani  \cite{MR1754036} (see also, \cite[Theorem 3.4]{MR3328860}) and M. Toma \cite{MR1742334} of Camacho's generalization. More recently, Ortiz-Bobadilla, Rosales-Gonzales, and Voronin \cite{MR2747734} proved a refined version of Camacho-Sad's result where one obtain lower bounds for the number of separatrices depending on the reduction tree of the foliation and the quotients of eigenvalues of the singularities.

Here we present a  proof, due to de Almeida Santos \cite{MR4094645}, which draws inspiration from Weil's residue theorem. The strategy has points of contact with the ones used by other authors, like Toma and Bobadilla-Gonzales-Voronin.

\begin{proof}[Proof of Theorem \ref{T:CamachoSad} according to \cite{MR4094645}]
    Let $\pi : (X,E) \to (\mathbb C^2,0)$ be a reduction of singularities of $\F$ given by Seidenberg theorem (dimension two version of Theorem \ref{T:reductionofsings}).
    Here $E$ is the reduced divisor supported on $\pi^{-1}(0)$. Clearly, $E$ is supported on a tree of rational curves.

    If one of the irreducible components of $E$ is not invariant
    by $\pi^* \F$, then we are done since we have infinitely many germs of curves invariant by $\pi^*\F$ and intersecting transversely one of the non-invariant components of $E$. The morphism $\pi$
    projects any of these germs of curves to a separatrix for $\F$.

    If all the irreducible components of $E$ are invariant then  we will show the existence
    of a smooth point $p$ of $E$ such that $\CS(\F,E,p) \neq 0$. Once this is done, the theorem follows from Briot-Bouquet's results.
    Aiming at a contradiction, assume that every singular point of $\pi^* \F$ on $E$ is either a corner (intersection of two irreducible components
    of $E$)  or a smooth point of $E$ such that $\CS(\pi^*\F,E,p)=0$ (so, $E$ is  a strong separatrix of a saddle node at $p$).

    Consider the dual graph $\Gamma$ of $E$. The vertices of $\Gamma$ are the irreducible components of $E$. Two vertices are joined by an edge if, and only if, the corresponding
    irreducible components intersect. We claim that there exists a connected subgraph $\Gamma_0$ of $\Gamma$ such that every edge in $\Gamma_0$ corresponds to a simple singularity of $\pi^*\F$ with invertible linear part. If all the edges of $\Gamma$ correspond to singularities with invertible linear part, the claim is evident as we can take $\Gamma_0=\Gamma$.
    If not, we delete from $\Gamma$ an edge corresponding to two irreducible components intersecting at a saddle-node. The resulting graph has two connected components. Keep the connected component containing the divisor corresponding to the strong separatrix of the saddle-node
    (in particular, the Camacho-Sad index of $\pi^*\F$ for this divisor at the saddle node is zero). Starting the argument over with this connected component in the place of $\Gamma$ will eventually lead to the sought subgraph $\Gamma_0$.

    The foliation $\pi^*\F$ determines an element of $H^1(\Gamma_0, \mathbb C^*)$  as follows. Take an open-covering $\mathcal U$ of $\Gamma_0$ formed  by open subsets indexed by vertices of $\Gamma_0$ and equal to the union of the corresponding vertex with all the edges containing it (opposite vertices removed). The intersection of any two of these open sets is empty, or the open set itself, or an edge (vertices removed) joining two vertices. This is a Leray-covering for the sheaf $\mathbb C^*$ on $\Gamma_0$. The foliation defines an element $\sigma$ of $C^1(\mathcal U, \mathbb C^*)$
    by setting
    \[
        (\sigma)_{CD} = - \CS (\pi^* \F,C,C\cap D)
    \]
    where $C$ and $D$ are two $\pi^*$-invariant curves supported on $E_0$ (the divisor determined by the vertices in $\Gamma_0$). Since every edge in $\Gamma_0$
    corresponds to a simple singularity with invertible linear part, a simple local computation shows that
    \[
        (\sigma)_{CD}  = - \CS (\pi^* \F,C,C\cap D) = \frac{1}{-\CS(\pi^* \F,D,C\cap D)} = \frac{1}{ (\sigma)_{DC}} \, .
    \]
    It follows that $\sigma$ determines an element of $Z^1(\mathcal U,\mathbb C^*)$ and hence of $H^1(\Gamma_0,\mathbb C^*)$. Since $\Gamma$ is a tree, so
    is $\Gamma_0$. Thus, $H^1(\Gamma_0,\mathbb C^*)= H^1(\mathcal U,\mathbb C^*)$ is the trivial multiplicative group $1$. Therefore, we get the existence of a  divisor $R \in \Div(X) \otimes \mathbb C$, supported on $E_0$,
    \[
        R = \sum_{C\in \Gamma_0} \lambda_C C
    \]
    such that for any two curves/vertices in $\Gamma_0$ we have that
    \[
        (\sigma)_{CD} = \frac{\lambda_{C}}{\lambda_D} \, .
    \]
    Theorem \ref{T:CamachoSad} implies that
    \begin{align*}
        R \cdot D & = {\lambda_D} \left( \sum_{C \in \Gamma_0} \frac{\lambda_C}{\lambda_D} C \right) \cdot D
                 =  {\lambda_D} \left( D^2 + \left( \sum_{C \in \Gamma_0, C \neq D} \frac{\lambda_C}{\lambda_D} C \right) \cdot D  \right) \\
                & = {\lambda_D} \left( D^2 - \sum_{p\in D} \CS(\pi^*\F,D,p)  \right) = 0
    \end{align*}
    for every curve/vertex $D$ in $\Gamma_0$. But this contradicts the negative definiteness of the intersection matrix of $E_0$ proving the result.
\end{proof}

The proof of Theorem \ref{T:CamachoSad} presented above also proves Camacho's generalization.  It is refined by de Almeida dos Santos to establish another instance of the separatrix theorem for foliations on singular surfaces, not covered by previous results. He proves that a germ of  foliation,on a singular surface,  with $\mathbb Q$-Cartier normal sheaf and without saddle-nodes in its reduction of singularities, always has a separatrix. A similar statement does not hold for foliations on singular surfaces with Cartier tangent sheaf. For an example, see  \cite{MR3330172}, where Guillot classifies complete vector fields on singular surfaces without a separatrix through one of its singular points.

The existence of separatrices for germs of one-dimensional foliations  on $(\mathbb C^n,0)$, $n \ge 3$, do not hold in general as shown by Gomez-Mont and Luengo in \cite{MR1172688}.

\subsection{Existence of separatrix for non-dicritical codimension one foliations}
The existence of separatrices for germs of codimension one foliations in dimension at least three
also does not hold in general. To explain that, we start recalling a result due to Jouanolou \cite[Chapter 4]{MR537038}.

\begin{thm}\label{T:Jouanolou}
    A very general foliation of $\mathbb P^2$ of degree $d\ge 2$ has no algebraic leaves.
\end{thm}

Theorem \ref{T:Jouanolou} admits many generalizations, see for instance \cite{MR1412680}, \cite{MR1971138}, \cite{MR2248154}, \cite{MR2862041}.
Even the original statement has now many different proofs exploiting different aspects of the theory of algebraic foliations
like automorphism groups of foliations \cite{MR1636411, MR2156709}), global geometry of the space of foliations (\cite{MR2391699}),
or arithmetic through Galois group actions (\cite{MR2364144})  to name a few.

A foliation on $\mathbb P^2$ of degree $d$  is defined by a polynomial homogeneous  $1$-form $\omega$  on $\mathbb C^3$. The algebraic
leaves of $\F$ are in bijection with the irreducible germs of hypersurfaces through $0 \in \mathbb C^3$ invariant by the foliation defined
by $\omega$ (the cone over $\F$). Theorem \ref{T:Jouanolou} implies the existence of (many) codimension one germs of foliations on $\mathbb C^3$
without separatrices. All these germs are dicritical: the reduction of singularities of each of them has at least one non-invariant irreducible exceptional divisor. For more on the concept of dicritical singularities, see \cite[Section 2.1]{MR1162557}, \cite{MR1010778}, and \cite{MR980953}.

In sharp contrast, we have the following result by Cano and Cerveau in dimension three  \cite{MR1179013}, and by Cano and Mattei in higher dimensions \cite{MR1162557}.

\begin{thm}\label{T:CanoCerveauMattei}
    If $\F$ is a non-dicritical  codimension one foliation on $(\mathbb C^n,0)$, $n\ge 3$, then $\F$ has a  separatrix.
\end{thm}

The proof by Cano and Cerveau shows more: if $\F$ is a non-dicritical foliation defined by a germ of $1$-form $\omega$ with
zeros of codimension at least two and $\gamma :(\mathbb C,0) \to (\mathbb C^3,0)$ is a parametrization of a germ of
curve not contained in $\sing(\F)$ such that $\gamma^* \omega \equiv 0$ then there exists a unique germ of surface $S$ through $0$
invariant by $\F$ and containing the image of $\gamma$, see \cite[Part IV, proof of Corollary 1.5]{MR1179013}, and \cite{MR1760842}.

Recently, Spicer and Svaldi  \cite{https://doi.org/10.48550/arxiv.1908.05037}, established the existence of separatrices for
germs of codimension one foliations with log canonical singularities on $(\mathbb C^3,0)$. We refer to their work for the definition
of log canonical singularities for foliations. Here, we just observe the existence of examples of dicritical foliations with log canonical singularities.

\section{Hodge theory and closed meromorphic 1-forms}\label{S:Hodge}

This section starts reviewing, in Subsections \ref{SS:closedness}, \ref{SS:decomposition}, and \ref{SS:logclosedness},
some elements of basic Hodge theory for compact Kähler manifolds. Subsections \ref{SS:realizalog} and \ref{SS:realizairregular}
explain how to construct closed meromorphic $1$-forms on compact Kähler manifolds with prescribed residue divisor and prescribed irregular
divisor, what will be used over, and over again, in the remainder of the paper.

\subsection{Closedness of holomorphic forms}\label{SS:closedness}
Recall that  a complex manifold $X$ is Kähler if it admits a smooth real
$(1,1)$-form $\theta$ such that $\theta$ is positive definite at every point
and $d \theta =0$. A $(1,1)$-form $\theta$ with these properties is called a Kähler form.

\begin{prop}\label{P:closedness}
    Let $X$ be a compact Kähler manifold and let $q\ge 0$ be a positive integer.
    If $\omega \in H^0(X,\Omega^q_X)$ is a holomorphic $q$-form on $X$ then $d \omega =0$.
\end{prop}
\begin{proof}
    Let $n$ be the dimension of $X$ and let $\theta$ be a Kähler form on $X$.
    Let  $\alpha$ be an arbitrary $p$-form and consider the $(n,n)$-form
    $\beta = \alpha \wedge \overline \alpha \wedge \theta^{n-p}$. If we choose local
    coordinates $z_1, \ldots, z_n$ then we can write
    \[
         \beta = \lambda b \cdot dz_1 \wedge d \overline{z_1} \wedge \cdots \wedge dz_n \wedge d \overline{z_n}
    \]
    where $\lambda$ is a suitable power of $i$ and $b$ is an everywhere non-negative function. Moreover, $b$ vanishes exactly at the zeros of $\alpha$. In particular,
    $\alpha$ is zero if, and only if,
    \[
        \int_X \beta = 0 \, .
    \]

    Set $p= q+1$, $\alpha= d \omega$ and $\beta= d\omega \wedge \overline {d\omega} \wedge \theta^{n-q-1}$.
    The closedness of $\theta$ implies that $\beta$ is the differential of $\omega \wedge \overline {d\omega} \wedge \theta^{n-q-1}$.
    Hence by Stoke's theorem
    \[
        \int_X d\omega \wedge \overline {d\omega} \wedge \theta^{n-q-1} = \int_{\partial X} \omega \wedge \overline {d\omega} \wedge \theta^{n-q-1} = 0 \, ,
    \]
    proving the proposition.
\end{proof}

\begin{remark}
    The argument above proves that the differential of holomorphic $q$-forms on
    compact complex manifolds of dimension $q+1$ are also closed. In particular, every
    holomorphic $1$-form on a compact complex surface is closed. Without further assumptions,
    not much more can be said concerning the closedness of holomorphic forms. Indeed, for every $q\ge 1$ and $r\ge 2$
    there exist compact complex manifolds of dimension $q+r$ admitting global holomorphic $q$-forms
    which are not closed. To verify this claim, start by considering a cocompact discrete subgroup $\Gamma$ of $\SL(2,\mathbb C)$.
    The quotient $X$ is a smooth compact complex $3$-fold carrying a non-closed holomorphic $1$-form, see for instance \cite[Example 2.13]{MR2324555}. Taking the product
    of $X$ with a compact complex torus of dimension $n-3$ one obtains compact complex manifolds having non-zero and non-closed
    holomorphic $k$-forms for every $k$ between $1$ and $n-2$.
\end{remark}

\subsection{Hodge decomposition}\label{SS:decomposition}
For a details, and proofs, of the results discussed in this subsection the reader
can consult for instance \cite[Chapter 6]{MR1988456} or \cite[Chapter 7]{MR1288523}.

Let $Z^1_X$ be the sheaf of closed holomorphic  $1$-forms.
If $X$ is compact then integration along elements of the first homology group $H_1(X,\mathbb Z)$ induces an injection
\[
    H^0(X, Z^1_X)  \oplus \overline{H^0(X,Z^1_X)}  \hookrightarrow H^1(X, \mathbb C) = \Hom(H_1(X,\mathbb Z), \mathbb C)\, .
\]
Indeed, elements of the kernel admit primitives that are (complex) pluriharmonic functions, hence necessarily constant (maximum principle)
on compact manifolds.

When $X$ is Kähler compact then Proposition \ref{P:closedness} guarantees $H^0(X,Z^1_X) = H^0(X,\Omega^1_X)$
and we have an inclusion of $H^0(X,\Omega^1_X)\oplus \overline{H^0(X,\Omega^1_X)}$ into $H^1(X,\mathbb C)$.
Moreover, every element of $H^1(X,\mathbb C)$ can be uniquely obtained as a sum of a morphism induced by the integration of a global holomorphic $1$-form with a morphism induced by the integration of the complex conjugate of another global holomorphic $1$-form. In other words, there exists
a natural isomorphism
\[
    H^0(X,\Omega^1_X) \oplus \overline{H^0(X,\Omega^1_X)} \longrightarrow H^1(X,\mathbb C)
\]
induced by integration.

Indeed, much more is true. For every $k \ge 0$, there exists a canonical isomorphism (the Hodge decomposition)
\[
    H^k(X,\mathbb C) = \bigoplus_{p+q=k} H^{p,q}_{\overline \partial}(X)  \, ,
\]
where $H^{p,q}_{\overline \partial}(X)$ denotes the Dolbeault cohomology of $X$. When $p$ and $q$ are both non-zero and $X$ is an arbitrary complex manifold, Dolbeault cohomology classes of degree $(p,q)$ do not define elements of $H^{p+q}(X, \mathbb C)$ in general. Nevertheless, on compact Kähler ambient spaces, every Dolbeault cohomology class  of type $(p,q)$ is represented by a  $(p,q)$-form that is not only $\overline \partial$-closed but also $d$-closed. The morphism  $H^{p,q}_{\overline \partial}(X) \to H^{p+q}(X,\mathbb C)$ implicit in the Hodge decomposition, sends a Dolbeault cohomology class $\alpha$ to the integration of a $d$-closed representative of $\alpha$  along $(p+q)$-cycles.

Alternatively, taking into account Dolbeault's isomorphisms $H^{p,q}_{\overline \partial}(X) \simeq H^q(X,\Omega^p_X)$, we can write
\[
    H^k(X,\mathbb C) = \bigoplus_{p+q=k} H^{q}(X, \Omega^p_X)  \, ,
\]
The existence of a unique $d$-closed representative for Dolbeault cohomology classes translates into the fact that the natural morphisms
\begin{equation}\label{E:iso closed}
    H^q(X,Z^p_X) \to H^q(X,\Omega^p_X)
\end{equation}
from the cohomology of the sheaf of closed holomorphic $p$-forms to the cohomology of the sheaf of holomorphic $p$-forms are surjective when
$X$ is a compact Kähler manifold.

\subsection{Closedness of logarithmic $1$-forms}\label{SS:logclosedness}

The result below is a direct consequence of the Hodge decomposition and
we learned it from Marco Brunella, see  \cite[Proposition 7.1]{MR2652221}.

\begin{prop}\label{P:closedness bis}
    Let $\omega$ be a meromorphic $1$-form on a compact Kähler manifold $X$. If
    $d \omega$ is holomorphic then $\omega$ is closed, \ie $d\omega =0$.
\end{prop}
\begin{proof}
    Notice that although
    $\omega$ is not closed a priori, the holomorphicity of $d\omega$ ensures that its residues along
    the irreducible components $H_i$ of its polar set are well-defined complex numbers. If
    $\sigma$ is a $2$-cycle  then Stoke’s Theorem implies that
    \[
        \int_{\sigma} d\omega = \sum \Res_{H_i}(\omega) \cdot  \int_{\sigma} c(H_i)
    \]
    It follows that the class of $d\omega$ in $H^2(X,\mathbb C)$ is a linear combination of the Chern classes
    of the hypersurfaces $H_i$. In particular, it is represented by an element of $H^{1}(X, \Omega^1_X) \subset H^2(X,\mathbb C)$.  At the same time, $d\omega$ is  a closed holomorphic 2-form, and therefore its class lies  in $H^0(X,\Omega^2_X)$.
    But the subspaces $H^1(X, \Omega^1_X)$ and $H^0(X, \Omega^2_X)$ of $H^2(X,\mathbb C)$ intersect only at zero since $X$ is Kähler compact.
    It follows that $d \omega=0$ as claimed.
\end{proof}

\begin{cor}\label{P:log closed}
    Let $D$ be a simple normal crossing divisor on a compact Kähler manifold. If $\omega \in H^0(X, \Omega^1_X(\log D))$ then $d \omega =0$.
\end{cor}
\begin{proof}
    It follows from the Exact Sequence (\ref{E:residue exact}) that the residues of $\omega$ along the irreducible components of $D$ are
    constant. A local computation shows that $d \omega$ is holomorphic and we can apply Proposition \ref{P:closedness bis} to conclude.
\end{proof}

The result above is a particular case of a theorem by Deligne which says that every logarithmic $p$-form  with poles on a simple normal crossing divisor on a compact Kähler manifold is  closed, \cite[Corollary 3.2.14]{MR498551}.

\subsection{Logarithmic $1$-forms with prescribed residues and $1$-forms of the second kind} \label{SS:realizalog}
Our next  result below, already mentioned in Subsection \ref{SS:residue},
is due to Weil and appears in the same work \cite{Weil47} where he proved the residue theorem.

\begin{prop}\label{P:residueWeil}
    If $X$ is a compact Kähler manifold then a divisor $R \in \Div(X)\otimes \mathbb C$
    with zero Chern class is the residue divisor of a closed logarithmic $1$-form.
\end{prop}
\begin{proof}
    Let $\mathcal U= \{ U_i\}$ be a sufficiently fine
    open covering of $X$ and let $\eta_i$ be closed logarithmic $1$-forms such that
    $\Res(\eta_i) = \restr{R}{U_i}$. The cocycle $\{ \eta_i - \eta_j\} \in Z^1(\mathcal U, \Omega^1_X)$
    represents the Chern class of $R$ in $H^1(X,\Omega^1_X) \subset H^2(X,\mathbb C)$ as explained in the proof of Proposition \ref{P:residueconstant}. Since this class is zero by assumption, there exists $\{ \omega_i\} \in C^0(\mathcal U, \Omega^1_X)$
    such that $\eta_i - \omega_i = \eta_j - \omega_j$ over non-empty intersections $U_i \cap U_j$. This shows the existence of a logarithmic $1$-form $\eta$ with constant residues and residue divisor equal to $R$. Proposition \ref{P:closedness bis} implies that $\eta$ is closed.
\end{proof}

Adhering to the classical terminology on the subject, we will say that a closed meromorphic $1$-form $\omega$
is of the second kind if $\Res(\omega)=0$.

\begin{cor}
    Let $X$ be a compact Kähler manifold and let $\omega$ be a closed meromorphic $1$-form
    on $X$. Then there exists a unique closed logarithmic $1$-form $\omega_{\log}$ and
    a unique closed meromorphic $1$-form $\omega_{\II}$ of the second kind with anti-holomorphic
    periods such that
    \[
        \omega = \omega_{\log} + \omega_{\II} \, .
    \]
\end{cor}
\begin{proof}
    Let $R= \Res(\omega)$ be the residue divisor of $\omega$. According to Theorem \ref{T:residue},
    $R$ has zero Chern class. Proposition \ref{P:residueWeil} guarantees the existence of a closed
    logarithmic $1$-form $\eta$ with $R=\Res(\eta)$. The difference $\omega - \eta$ is a closed
    meromorphic $1$-form of the second kind. As such, it represents a class in $H^1(X,\mathbb C)$.
    By subtracting the holomorphic component of $\omega - \eta$ from it, we may assume that it has
    anti-holomorphic periods and set $\omega_{\II}$ equal to it. To conclude it suffices to take $\omega_{\log} = \omega - \omega_{\II}$.
\end{proof}

\subsection{Closed meromorphic $1$-forms with prescribed irregular divisor}\label{SS:realizairregular}
We now present the result mentioned  in Subsection \ref{SS:irregular}, which characterizes the irregular divisors
of closed meromorphic $1$-forms with simple singularities.

\begin{thm}\label{T:realizaII}
    Let $X$ be a compact Kähler manifold and let $I \in \Div(X)$ be a simple normal crossing effective divisor satisfying the conclusion of Proposition  \ref{P:irregular}:
    \[
        \mathcal O_X(I) \otimes \frac{\mathcal O_X}{\mathcal O_X(-I)} \simeq \frac{\mathcal O_X}{\mathcal O_X(-I)} \, .
    \]
    Then there exists a closed meromorphic $1$-form $\omega$ of the second kind and without base points such that $(\omega)_{\infty} = I + I_{red}$.
\end{thm}
\begin{proof}
    The assumption implies the existence of an open covering $\{U_i\}$ of $X$, holomorphic functions $f_i \in \mathcal O_X(U_i)$ defining $I$,
    and holomorphic functions $r_{ij} \in \mathcal O_X(U_i \cap U_j)$
    such that $f_{i} = \left( 1 + f_j \cdot r_{ij} \right) f_j$.
    This implies that the differences $a_{ij} = \frac{1}{f_i} - \frac{1}{f_j}$
    are holomorphic functions and determine a class $\{a_{ij}\}$ in $H^1(X, \mathcal O_X)$. The surjectivity of the morphisms presented in Equation  (\ref{E:iso closed})
    implies the existence of closed holomorphic $1$-forms $\alpha_i \in \Omega^1_X(U_i)$ such that $d a_{ij} = \alpha_i - \alpha_j$,
    see the proof of \cite[Theorem B]{MR3940902}. The closed meromorphic $1$-forms of the second kind
    \[
        d \left( \frac{1}{f_i} \right) - \alpha_i
    \]
    coincide at non-empty intersections $U_i \cap U_j \neq \emptyset$ and define the sought $1$-form.
\end{proof}

\begin{remark}\label{R:thmB}
    Theorem \ref{T:realizaII} is a particular case of \cite[Theorem B]{MR3940902}, see also \cite[Theorem B]{MR4333432}. These results, imply
    that if
    \[
        \mathcal O_X(I) \otimes \frac{\mathcal O_X}{\mathcal O_X(-I)} \simeq \mathcal L \otimes \frac{\mathcal O_X}{\mathcal O_X(-I)} \, .
    \]
    for some line-bundle $\mathcal L \in \Pic(X)$ with zero Chern class in $H^2(X,\mathbb Q)$ then there exists a global  meromorphic
    $1$-form $\omega$ with coefficients in $\mathcal L^*$, polar divisor equal to $I+I_{red}$, no residues, and closed with respect to the
    unitary connection on $\mathcal L^*$. The proof is essentially the same.
\end{remark}

\section{Polar divisor of logarithmic 1-forms}\label{S:logarithmic}

\subsection{Hodge index theorem}
Let $X$ be a compact Kähler manifold of dimension $n$ and let $\theta$ be a Kähler form.
Consider the symmetric bilinear form on
$H^2(X,\mathbb R)$ defined by
\[
    (\alpha, \beta)_{\theta} = \int_X \alpha \wedge \beta \wedge \theta^{n-2}
\]
for any $\alpha, \beta \in H^2(X,\mathbb R)$.

\begin{lemma}\label{L:trivial}
    Let $X$ be a compact Kähler manifold with Kähler form $\theta$. If $D_1, D_2 \in \Div(X)$ are  non-zero effective divisors with disjoint
    supports such that $(c(D_1),c(D_1))_{\theta} \neq 0$ then $c(D_1)$ and $c(D_2)$ are linearly independent in $H^2(X,\mathbb R)$.
\end{lemma}
\begin{proof}
    The effectiveness of $D_i$ implies that $(c(D_i),\theta)_\theta = \int_{D_1} \theta^{n-1} \neq 0$. Therefore both $c(D_1)$ and $c(D_2)$ are non-zero
    elements of $H^2(X,\mathbb R) \cap H^{1,1}(X)$. Since the supports of $D_1$ and $D_2$ are disjoint, we have that $(c(D_1),c(D_2))_{\theta}=0$.  If $c(D_1)$ and  $c(D_2)$ are linearly dependent then we would get $c(D_2) = \lambda c(D_1)$ for some $\lambda \neq 0$. This leads to the
    contradiction $0 = (c(D_1),c(D_2))_\theta = \lambda (c(D_1), c(D_1))_\theta \neq 0$.
\end{proof}

The Hodge index theorem \cite[Theorem 6.33]{MR1988456} implies that the restriction
of the symmetric form $(\cdot, \cdot)_{\theta}$ to $H^2(X,\mathbb R) \cap H^{1,1}(X)$ is non-degenerate and has signature equal to $(1,h^1(X,\Omega^1_X)-1)$, \ie one positive eigenvalue and $h^1(X,\Omega^1_X)-1$ negative eigenvalues.

\begin{dfn}
    Let $R \in \Div(X)$ be  a divisor (or a $\mathbb C$-divisor) on compact Kähler manifold $X$ and let $H_1, \ldots, H_k$ be the irreducible components of the support of    $R$, \ie
    \[
        R = \sum_{i=1}^k \lambda_i H_i
    \]
    with $\lambda_i \neq 0$ for every $i$. The intersection matrix of $R$ (with respect to $\theta$) is, by definition, the $k \times k$ symmetric matrix with entry $(i,j)$ equal to $(c(H_i), c(H_j))_{\theta}$.
\end{dfn}

\begin{lemma}\label{L:Hodgeindex}
    Let $X$ be a compact Kähler manifold with a Kähler form $\theta$. Let $R = \sum_{i=1}^k R_i \in \Div(X) \otimes \mathbb C$ be a divisor which is a sum of $k\ge 1$ non-zero divisors $R_i \in \Div(X) \otimes \mathbb C$ with connected and pairwise disjoint supports.
    If $c(R) =0$ then, for every $i$, there exists a non-zero effective divisor $D_i \in \Div(X)$ (integral coefficients) such that
    $(D_i,D_i)_{\theta} \ge 0$ and with support contained in the support of $R_i$. Moreover, we have the following alternative.
    \begin{enumerate}
        \item\label{I:Hodge 1} If the intersection matrix of $R$ is not semi-definite negative then the support of $R$ is connected.
        \item\label{I:Hodge 2} Otherwise, the support of $R$ is not connected, the number of zero eigenvalues of the intersection matrix
         of $R$ is exactly the number $k$ of connected components of the support of $R$, and there exists complex numbers $\lambda_i$ such that $R_i = \lambda_i D_i$.
    \end{enumerate}
\end{lemma}
\begin{proof}
    Let $H_1, \ldots, H_{\ell}$ be the irreducible components of the support of $R$. Consider the
    $\mathbb Q$-vector space $V = \oplus_{i=1}^{\ell} \mathbb Q \cdot H_i \subset \Div(X) \otimes \mathbb Q$. Taking Chern classes defines
    a linear map $\varphi : V \otimes \mathbb C \to  H^2(X,\mathbb C)$. By assumption $R \in \ker \varphi$. Since
    $\varphi$ is defined over $\mathbb Q$ we obtain the existence of a divisor $D \in \Div(X)$ with zero Chern class and with
    support equal to the support of $R$.

    Write $D = \sum_{i=1}^k (D_{i,+} - D_{i,-})$, where, for each $i$, $D_{i,+}$ and $D_{i,-}$ are effective divisors with supports contained in the support of $R_i$ and without common irreducible components. Since $c(D)=0$, Lemma \ref{L:trivial} implies that
    $(D_{i,+}-D_{i,-},D_{i,+}-D_{i,-})_{\theta} =0$ for every $i$. If we take $D_i = D_{i,+} + D_{i,-}$ then
    \[
        (D_i,D_i)_{\theta} = (D_{i,+}-D_{i,-},D_{i,+}-D_{i,-})_{\theta} + 4 (D_{i,+}, D_{i,-})_{\theta} \ge 0
    \]
    because $D_{i,+}$ and $D_{i,-}$ are effective divisors without common irreducible components.

    Assume that the intersection matrix of $R$ is not negative semi-definite. By definition, there exists a divisor $E$ contained in the support of $R$ such that $(E,E)_{\theta} >0$.
    Without loss of generality, we can assume that $E$ has connected support. If the support of $R$ were not connected
    then we would have two non-zero pairwise disjoint divisors, say $E$ and $D_j$, with linear independent Chern classes
    (Lemma \ref{L:trivial}) generating a two dimensional vector subspace of  $H^2(X,\mathbb R) \cap H^{1,1}(X)$
    where the intersection form has two non-negative eigenvalues contradicting the Hodge index theorem. This shows Item (\ref{I:Hodge 1}).

    From now on, assume that $(E,E)_{\theta}\le 0$ for every divisor $E$ with support contained in the support of $R$.
    In particular, $(D_i,D_i)_{\theta}=0$ for every $i$. Since $D_i$ is effective, the Chern class
    of $D_i$ is non-zero. Therefore, the support of $D$, and hence of $R$, must have at least two distinct connected components.
    Notice that for $i \neq j$, Hodge index theorem  implies that $D_i$ and $D_j$ have proportional Chern classes. Likewise
    the real and the imaginary parts of $R_i$ must have Chern classes proportional to the Chern class of $D_j$.
    It follows the existence of complex numbers $\lambda_i$ such that $c(R_i) = \lambda_i c(D_i)$.

    We now proceed to show that $R_i = \lambda_i D_i$ as divisors. Assume that this is not the case.
    Considering linear combinations of $D_i$ with the real or imaginary part of $R_i$, we can produce a non-zero $\mathbb R$-divisor $E_i$ with support contained in the support of $R_i$ but not equal it and such that $(E_i,E_i)_{\theta}=0$. Since the support of $R_i$ is connected, there exists $H$ contained in it such that $(E_i, H)_{\theta} > 0$. If we take  $E = \lambda E_i +  H$ for some $\lambda \gg 0$ then $(E,E)_{\theta} > 0$ contradicting our assumption. It follows that $R_i = \lambda_i D_i$ as claimed. Notice that we have also proved that the only divisors of zero self-intersection and support connected and contained in $|R|$ are multiples of one of the divisors $D_i$. Hence $k$ is the number of connected components of $R$ as claimed.
\end{proof}

\begin{lemma}\label{L:BeauvilleTotaro}
    Let $X$ be a compact Kähler manifold that has a map $f : X \to C$
    with connected fibers onto a smooth curve. Then any nonzero effective divisor $D$
    on $X$ such that $(D, D)_{\theta} = 0$, for some Kähler form $\theta$, which maps to a point $p \in C$
    is a positive rational multiple of the divisor $f^{-1}(p)$.
\end{lemma}
\begin{proof}
    The statement, when $X$ is projective,  appears as \cite[Lemma 3.1]{Totaro}.
    The very same proof works when $X$ is a compact Kähler manifold.
\end{proof}

\subsection{Logarithmic $1$-form canonically attached to a divisor with zero Chern class}
Given a divisor $D \in \Div(X) \otimes \mathbb C$ with zero Chern class on a compact Kähler manifold, Proposition \ref{P:residueWeil} gives the existence of a closed logarithmic $1$-form $\omega$ such that $\Res(\omega)= D$. Of course, $\omega$ is not unique as we may replaced
it by $\omega + \eta$ for any $\eta \in H^0(X,\Omega^1_X)$.

\begin{prop}\label{P:omegaD}
    Let $D \in \Div(X) \otimes \mathbb R$ be a divisor on a compact Kähler manifold $X$.
    If $c(D) = 0$ then there exists a unique closed logarithmic $1$-form $\omega_D$ such that
    $\Res(\omega_D) = D$ and the periods of $\omega_D$ are purely imaginary complex numbers.
\end{prop}
\begin{proof}
    We start by verifying uniqueness. If $\omega_D$ and $\tilde{\omega_D}$ are two logarithmic $1$-forms
    with purely imaginary periods and same residues then $\eta = \omega_D - \tilde{\omega_D}$ is a holomorphic $1$-form
    with imaginary periods. Therefore, for any $x_0 \in X$,  the modulus of $\exp \int_{x_0}^x \eta$ is a well-defined plurisubharmonic function  $F : X \to \mathbb R$. The compactness of $X$ implies that $F$ attains a maximum. The  maximum principle implies that $F$ is constant and, consequently, that $\eta$ vanishes identically. The uniqueness follows.

    For the existence, first assume that $D$ has integral coefficients. In this case, the line-bundle $\mathcal O_X(D)$ admits
    a unique flat unitary connection. This connection defines  a closed logarithmic $1$-form $\Omega$ on the total space of
    $\mathcal O_X(D)$ as explained in the proof of Theorem \ref{T:residue for connections}. The fact that the connection is unitary implies
     that the periods of $\Omega$ are purely imaginary. If $\sigma$ is a meromorphic section
    of $\mathcal O_X(D)$ such that $(\sigma)_0 - (\sigma)_{\infty} = D$ then we can take $\omega_D$ as $\sigma^* \Omega$. If $D$ is arbitrary
    divisor with real coefficients then we can write $D = \sum \lambda_i D_i$ where $\lambda_i \in \mathbb R$ and $D_i \in \Div(X)$ are
    (integral) divisors with zero Chern classes. The logarithmic $1$-form $\sum \lambda_i \omega_{D_i}$ has the sought properties.
\end{proof}

\begin{prop}\label{P:invariante}
    Let $D \in \Div(X) \otimes \mathbb R$ be a divisor on a compact Kähler manifold with $c(D)=0$ and $D\neq 0$.
    \begin{enumerate}
        \item\label{I:a} If $V$ is compact complex variety and $i : V \to  X$ is a holomorphic map such that $i(V) \cap |D| = \emptyset$ then
        $i^*\omega_D = 0$.
        \item\label{I:b} If $\G$ is a codimension one foliation on $X$ and $L$ is a leaf of $\G$ such that its topological closure  $\overline L$ does not intersect $|D|$ then either $\overline L$ is a compact analytic hypersurface or  $\G$ is equal to the foliation defined by $\omega_D$.
    \end{enumerate}
\end{prop}
\begin{proof}
     Let $\F_{D}$ be the foliation defined by $\omega_{{D}}$. Since the periods of
     $\omega_D$ are purely imaginary, the function
    \begin{align*}
        F_{D} : X - |D| & \longrightarrow (0, \infty) \\
        x & \mapsto \left\vert \exp \int^x \omega_{{D}} \right\vert
    \end{align*}
    is a plurisubharmonic first integral for $\restr{\F_{D}}{X - |D|}$.

    To verify Item (\ref{I:a}), it suffices to observe that $i^* F_D$ admits a maximum and, hence, by the maximum
    principle must be constant. It follows that $i^* \omega_D$ vanishes identically.

    The proof of Item (\ref{I:b}) relies on the same idea. Assume that $L$ is not contained in any
    closed hypersurface of $X - |D|$. Since $F_D$ is locally the modulus of a holomorphic function, the restriction of $F_D$ to $L$ is either constant or an open map. If  $\restr{F_D}{L}$ is constant then $L$ is a leaf of both $\F_D$ and $\G$ and hence $\F_D$ and $\G$ must be tangent
    on $\overline L$. Since $\overline L$ is not contained on a compact analytic hypersurface, $\F_D$ and $\G$ must coincide.
    
    Assume from now on that $\restr{F_D}{L}$ is not constant. Since $L$ is connected and $\restr{F_D}{L}$ is open, it follows that
    $F_D(L)$ is an interval $(a,b) \subset (0,\infty)$ with compact closure. Therefore $F_D(\overline L) = [a,b]$. Let $\pi: Y \to X$ be the universal covering of $X$ and consider the function
    \begin{align*}
        f_{D} : Y  & \longrightarrow \mathbb P^1 \\
        x & \mapsto  \exp \int^x \omega_{{D}} \, .
    \end{align*}
    Let  $M =\pi^{-1}(L)$ be the pre-image of $L$ on the universal covering. Since $\overline{M}$ is not
    compact, it is not clear that $f_D(\overline{M})$  is closed. But since
    a fundamental domain of the covering $\pi:Y \to X$ has compact closure in $Y$, we have that
    \[
        \bigcup_{\gamma \in \pi_1(X)} \exp\left( \int_{\gamma} \omega_D \right) \cdot f_D(\overline{M})
    \]
    coincides with $\overline{f_D(M)}$. Since $\pi_1(X)$ is countable and the boundary of $\overline{f_D(M)}$ is uncountable, we deduce the existence of uncountably many leaves of $\pi^*\G$, contained in $\overline M-M$, where $f_D$ is constant. It follows that $\F_D$ and $\G$ must coincide. For details, see the proof of \cite[Proposition 5.1]{PereiraJAG}.
\end{proof}

\subsection{A criterion for the existence of fibrations}
The result below is a version of a \cite[Theorem 2.1]{Totaro} due to Totaro, see also \cite[Theorem 3]{PereiraJAG}.

\begin{thm}\label{T:fibers}
    Let $D \in \Div(X) \otimes \mathbb C$ be a divisor with zero Chern class on a compact Kähler manifold $X$. If the support of $D$ has $c\ge 3$ distinct connected components then there exists a morphism $f:X \to C$ to a curve such that $D \in f^* \Div(C) \otimes \mathbb C$.
\end{thm}
\begin{proof}
    Lemma \ref{L:Hodgeindex} implies that $D  = \sum_{i=1}^c \lambda_i D_i$ where $D_1, \ldots, D_c$ are effective divisors with  connected, pairwise disjoint supports and proportional Chern classes. After replacing the $D_i$'s by appropriate multiples, we may assume that $c(D_i) = c(D_j)$ for any $i,j$. For $i\neq j$, let $D_{ij} = D_i - D_j$ and consider the closed logarithmic $1$-form $\omega_{ij} = \omega_{D_{ij}}$ given by Proposition \ref{P:omegaD}.

    Let $\F_{ij}$ be the foliation defined by $\omega_{{ij}}$. The function
    \begin{align*}
        F_{ij} : X - |D_i+D_j| & \longrightarrow (0, \infty) \\
        x & \mapsto \left\vert \exp \int^x \omega_{{ij}} \right\vert
    \end{align*}
    is a  plurisubharmonic first integral for $\restr{\F_{ij}}{X - |D_i+D_j|}$. Since $|D_i| \cap |D_j| = \emptyset$, the level sets of
    $F_{ij}$ are compact subsets of  $X - |D_i + D_j|$. If $1\le i<j<k \le c$ are three distinct indices then the maximum principle implies that level sets of $F_{ij}$ sufficiently close to $D_j$ must be contained in level sets of $F_{ik}$. This is sufficient to show that the foliations $\mathcal F_{ij}$ and $\mathcal F_{ik}$ coincide. It follows the
    existence of a non-constant meromorphic function $h$ such that $\omega_{ij} = h \omega_{ik}$. The Stein factorization of $h : X \to \mathbb P^1$  gives the morphism
    $f: X \to C$. The existence of a divisor on $C$ such that $D$ is the pull-back of it follows from Lemma \ref{L:BeauvilleTotaro}.
\end{proof}

\subsection{Poles of logarithmic $1$-forms}\label{SS:connected components}
The Hodge index theorem, or more specifically Lemma \ref{L:Hodgeindex}, has strong implications on the polar set of closed logarithmic $1$-forms.

\begin{thm}\label{T:structure log forms}
    Let $\omega$ be a closed logarithmic $1$-form on a compact Kähler manifold $X$.
    Assume that $\Res(\omega) \neq 0$ and let $\mathfrak c(\omega)$ be the number of connected components
    of the support of $\Res(\omega)$. Then
    \begin{enumerate}
        \item\label{I:structure 1} The intersection matrix of $\Res(\omega)$ is negative semi-definite
    if, and only if, $\mathfrak c(\omega) \ge 2$.
        \item\label{I:structure 2} If $\mathfrak c(\omega) \ge 3$ then there exists a morphism $f : X \to C$ to a compact curve $C$
        and a logarithmic $1$-form $\alpha$ on $C$ such that $\omega - f^*\alpha \in H^0(X,\Omega^1_X)$.
        \item\label{I:structure 3} If $\mathfrak c(\omega) =2$ then there exists a complex number $\lambda \in \mathbb C^*$ such $\Res(\lambda \omega) \in \Div(X)$ (\ie $\Res(\lambda \omega)$  has integral coefficients) and there exists a holomorphic $1$-form $\beta \in H^0(X,\Omega^1_X)$ such that the periods
            of $\lambda \omega - \beta$ are purely imaginary. Moreover, if the period group of $\lambda \omega - \beta$  is
    a discrete subgroup of $(i\mathbb R, +)$ then, as when $\mathfrak c(\omega) \ge 3$, there exists a morphism $f:X \to C$ to a compact curve $X$ and logarithmic $1$-form $\alpha$ on $C$ such that $\omega - f^* \alpha$ belongs to $H^0(X,\Omega^1_X)$.
    \end{enumerate}
\end{thm}
\begin{proof}
    Item (\ref{I:structure 1}) follows from the residue theorem (Theorem \ref{T:residue}) combined with Item (\ref{I:Hodge 1}) of Lemma \ref{L:Hodgeindex}. Item (\ref{I:structure 2})
    follows from Theorem \ref{T:fibers} combined with the residue theorem and the existence of logarithmic $1$-forms with prescribed residues (Proposition \ref{P:residueWeil}).
    The existence of $\lambda$ as claimed in Item (\ref{I:structure 3}) follows from Item (\ref{I:Hodge 2}) of Lemma \ref{L:Hodgeindex}. The existence of $\beta$ follows from Proposition \ref{P:omegaD}.
    If the period group of $\lambda \omega - \beta$ is discrete then, after multiplying $\lambda \omega - \beta$ by a suitable integer, we can assume that it is contained
    in $2 \pi \sqrt{-1} \mathbb Z$. Hence $\exp \int \left( \lambda \omega - \beta \right): X \to \mathbb P^1$ is a well defined morphism. If $f: X \to C$ is the Stein factorization of this
    morphism then the existence of $\alpha$ follows from Lemma \ref{L:BeauvilleTotaro} and the existence of logarithmic $1$-forms with prescribed residues.
\end{proof}

\subsection{Topology of hypersurfaces}
In this section, we discuss a result on the topology of smooth hypersurfaces of compact Kähler manifolds whose proof
is based on the study of the foliation defined by  a certain logarithmic $1$-form.

\begin{thm}
    Let $X$ be a compact Kähler manifold and let $Y_1, Y_2 \subset X$ be two disjoint smooth hypersurfaces
    such that $a \cdot c(Y_1) = b \cdot c(Y_2)$ in $H^2(X,\mathbb Q)$ for suitable integers $a$ and $b$. Then there
    exists finite étale coverings $p_i : Z_i \to Y_i$ such that $\deg(p_1) b = \deg(p_2) a$ and
    $Z_1$ is $C^{\infty}$-diffeomorphic to $Z_2$.
\end{thm}
\begin{proof}
    We will briefly sketch the proof. For details, we refer to \cite{PereiraJAG}.

    Let $D$ be the divisor $a Y_1 - bY_2$. By assumption, the Chern class of $D$ is zero. Let $\omega_D$ be the $1$-form given
    by Proposition \ref{P:omegaD}, \ie $\omega_D$  is the unique logarithmic $1$-form with  $\Res(\omega_D)=D$ and purely imaginary periods.

    If the periods of $\omega_D$ are contained in $i\pi \mathbb Q$, then for a suitable $N \in \mathbb N$ we have that all
    the periods of $N\omega_D$ are integral multiples of $2 \pi i$ and hence $F = \exp(\int N \omega_D): X \to \overline{\mathbb C}$
    is a well-defined holomorphic map with fibers over $0$ and $\infty$ equal to suitable multiples of $Y_1$ and $Y_2$. In this particular
    case the result follows by taking, $Z_i = F^{-1}(z_i)$ equal to a smooth fiber near $Y_i$. Indeed, one can choose a path $\gamma$ on $\overline C$
    joining $z_1$ to $z_2$ and avoiding all the finitely many critical values of $F$. Ehresmann fibration theorem implies that $Z_1$ and
    $Z_2$ are $C^{\infty}$-diffeomorphic. Finally, the local structure of the foliation defined by $\omega$ implies that $Z_1$ and $Z_2$
    are $C^{\infty}$-diffeomorphic to étale coverings of $Y_1$ and $Y_2$.

    If the periods of $\omega_D$ are not contained in $i \pi \mathbb Q$, then the idea is to perturb $\omega_D$ by adding
    a sufficiently small imaginary valued closed $1$-form $\eta$ such that $\omega_D + \eta$ has all its periods contained
    in $i \pi \mathbb Q$ and argue as in the previous paragraph. Some care must be taken since the critical
    values of a differentiable function $X \to \overline{\mathbb C}$ might disconnect the contra-domain and one might not
    be able to connect the smooth fibers nearby $Y_1$ to the smooth fibers nearby $Y_2$. To avoid this difficulty, one uses Lemma \ref{L:existence primitive}
    in order to show the existence of arbitrarily $1$-forms $\eta$ such that $\omega_D + \eta$ has periods in $i \pi \mathbb Q$ and
    $\eta$ is identically zero at a neighborhood of the zeros of $\omega_D$. By taking $F = \exp(\int N (\omega_D + \eta) ): X \to \overline{\mathbb C}$ one obtains a differentiable map with finitely many critical values, and one can apply the arguments of the previous paragraph.
\end{proof}

\subsection{Quasi-invariant hypersurfaces}
In this subsection, we will restrict to foliations on projective manifolds. Likely,  the results presented here are also
valid for foliations on compact Kähler manifolds, but the arguments presented in \cite{MR4150930} rely on taking hyperplane sections in order
to reduce the problem to dimension three where reduction of singularities is available.

Let $\F$ be a codimension one foliation on a projective manifold $X$. An irreducible hypersurface $H$ is quasi-invariant by
$\F$ if $H$ is not $\F$ invariant and the restriction of $\F$ to $H$ is algebraically integrable.

\begin{thm}\label{T:quasiJouanolou}
    Let $\F$ be a codimension one holomorphic foliation on a projective  manifold $X$.
    If $\F$ admits infinitely many quasi-invariant hypersurfaces then $\F$ is algebraically integrable, or $\F$ is a pull-back of a foliation on a projective surface under a dominant rational map.
\end{thm}

The proof of Theorem \ref{T:quasiJouanolou} can be briefly summarized as follows. First show that it suffices to prove the result for foliations with simple singularities on projective $3$-folds. Then construct divisors $D_1$ and $D_2$ supported on quasi-invariants hypersurfaces with zero Chern class and without common irreducible components in their supports. Let $\omega_{D_1}$, $\omega_{D_2}$ be the unique logarithmic $1$-forms with purely imaginary periods and $\Res(\omega_{D_i})=D_i$ given by Proposition \ref{P:omegaD}. Consider the restriction of the foliations defined by $\omega_{D_i}$ at a general leaf $L$ of $\F$. If the supports of $\restr{D_1}{L}$ and $\restr{D_2}{L}$ intersect, then $L$ must be algebraic since it is a surface contained in a projective $3$-fold and containing a divisor of positive self-intersection. In this case, $\F$ is algebraically integrable. If the supports of $\restr{D_1}{L}$ and $\restr{D_2}{L}$ do not intersect then, arguing as in the proof of Theorem \ref{T:fibers}, one can show that the leaves of these foliations are algebraic subvarieties of $L$, and is this way produce a codimension two foliation $\G$ be algebraic leaves tangent to $\F$. To properly carry out this argument, it is important that the foliation $\F$ has simple singularities. Standard arguments imply that $\F$ is the pull-back of foliation on a projective surface. We invite the reader to consult \cite{MR4150930} for details.

\section{Semi-global separatrices}\label{S:algebraicity}

In this section, contrary to the everywhere else int the paper, we choose to restrict to codimension one foliations on projective manifolds of dimension three, because most of our results  rely on the reduction of singularities for codimension one foliations. With some extra work, we could have  formulated  similar results for codimension one foliations on projective manifolds of any dimension  greater than three, and proved them by restricting the foliation to a sufficiently general projective $3$-fold. As this would add an extra layer of complexity to the proofs, and no
new idea, we opted to restrict to complex manifolds of dimension three. The hypothesis on the projectiveness of the ambient space is more serious, as most of the algebraicity results for subvarieties that we use rely on the existence of sufficiently many meromorphic functions on the ambient space, and an analog reasoning cannot be applied to non-algebraic compact Kähler manifolds.

\subsection{Semi-global separatrices}
Let $\F$ be a codimension one foliation on a complex manifold $X$ of dimension three defined by a twisted $1$-form $\omega \in H^0(X,\Omega^1_X \otimes N_F)$. Let $i: S \hookrightarrow X$ be the inclusion of an irreducible locally closed surface. We will say that $S$ is a semi-global separatrix for the foliation $\F$ if $\restr{(i^* \omega)}{S - \sing(S)}$ is identically zero and $i^{-1}(\sing(\F))$ is a connected compact subset of $S$.

\begin{prop}\label{P:extend sep}
    Let $\F$ be a codimension one foliation with simple singularities on a complex manifold $X$ of dimension three.
    Let $C$ be an irreducible curve contained in the singular set. If $C$ is compact then every point $p \in C - \sing(C)$ belongs
    to at least one, and at most two, semi-global separatrix of $\F$.
\end{prop}
\begin{proof}
    Let $p \in C$ be a smooth point of $C$. If $\Sigma \simeq (\mathbb C^2,0)$ is a germ of surface transverse to
    $C$ at $p$ then $\restr{\F}{\Sigma}$ is a germ of foliation on $(\mathbb C^2,0)$ with a simple singularity.
    Briot-Bouquet's theorem guarantees that $\restr{\F}{\Sigma}$ has at least one convergent separatrix, and at most two.
    The proof of the separatrix theorem for germs of non-dicritical codimension one foliations, see \cite[Part IV]{MR1179013}, \cite[Subsection 5.6]{MR1760842} and the discussion following the statement of Theorem \ref{T:CanoCerveauMattei}, implies the existence of a semi-global separatrix containing any
    given (convergent) separatrix of $\restr{\F}{\Sigma}$.
\end{proof}

Let $S$ be a semi-global separatrix for a foliation $\F$ with simple singularities on a complex manifold of dimension three.
Let $n: S^n \to S$ be the normalization of $S$. Note that $S^n$ is smooth since $S$ is normal crossing along the singular set of $\F$.
The pull-back of Bott's connection to $S^n$, \ie the connection defined in Subsection \ref{SS:index}, is a flat meromorphic connection $\nabla$ with polar divisor having support contained in the compact set $n^{-1}(\sing(\F))$. Let $C_1, \ldots, C_k$ be the irreducible components of $n^{-1}(\sing(\F))$. Since we are assuming that $\F$ has simple singularities, each $C_i$ is a curve and $\sum C_i$ is a normal crossing divisor on $S^n$. The intersection matrix of the semi-global separatrix $S$ is, by definition, $(C_i \cdot C_j)$, the intersection matrix of the divisor $\sum C_i$.

We will say that the semi-global separatrix $S$ is logarithmic if the connection $\nabla$ is logarithmic. Otherwise, if $\nabla$ has non-reduced polar divisor then we say that $S$ is an irregular semi-global separatrix.

Inspired by Theorem \ref{T:structure log forms}, we will now proceed to analyze the semi-global separatrices according to properties
of its intersection form.

\subsection{Negative definite intersection form} The key result concerning (germs of) smooth surfaces
containing a divisor with compact support and negative definite intersection form is Grauert's contractibility
criterion \cite[Section 8.e]{Grauert}.

\begin{thm}\label{thm:Grauert}
    Let $S$ be a germ of smooth complex surface containing a divisor
    $D$ with compact and connected support.  If the intersection form of $D$ is negative definite then there exists
    a germ of  normal complex surface $S'$ and a morphism $\pi: S \to S'$ which maps
    $|D|$ to a point $p$ and is an isomorphism between $S- |D|$ and $S' - \{p \}$.
\end{thm}

In particular,  the germ of  $S$ around $|D|$ has a huge ring of holomorphic functions. Compactness
of $|D|$ implies that all holomorphic functions are constant along $|D|$, but they separate points on
the complement $S-|D|$.

\begin{prop}\label{P:negative connection}
    Let $S$ be a germ of smooth complex surface containing a divisor
    $D$ with compact support and negative definite intersection form.
    If $\mathcal L$ is a line-bundle over $S$ and $\nabla$ is a flat meromorphic connection on $\mathcal L$ with polar divisor
    supported on $|D|$ then $\nabla$ is a logarithmic connection and $\Res(\nabla) \in \Div(S) \otimes \mathbb Q$.
\end{prop}
\begin{proof}
    Proposition \ref{P:irregularconnections} implies that the irregular divisor of
    $\nabla$ satisfies $I(\nabla)^2\ge 0$. Since the intersection form of $D$ is negative definite, $I(\nabla)$ must be zero, \ie $\nabla$ is logarithmic. If $C_1, \ldots, C_k$ are the irreducible components of the support of $D$ then the negative definiteness of $|D|$ combined with Theorem \ref{T:residue for connections} (residue theorem for connections) implies that the residue divisor of $\nabla$ is completely determined by the system of linear equations, defined over $\mathbb Z$,  $\Res(\nabla) \cdot C_i = - c(\restr{\mathcal L}{C_i})$, $i \in \{ 1, \ldots, k\}$. It follows that $\Res(\nabla)$ is a divisor with rational coefficients.
\end{proof}

\begin{cor}\label{C:Bottnegative}
    Let $\F$ be a codimension one foliation with simple singularities on a complex manifold $X$ of
    dimension three. If $S$ is a semi-global separatrix with negative definite intersection form then
    Bott's connection is a logarithmic connection with rational residues.
\end{cor}

\subsection{Intersection form with a positive eigenvalue}
We now turn our attention to semi-global separatrices with intersection forms having at least one positive eigenvalue.
The situation is  opposite to the case of semi-global separatrices with negative definite intersection forms.
While in the previous case, we had an abundance of holomorphic functions, in the present situation the only holomorphic
functions are constant and, moreover, the field of  meromorphic functions of the normalization of a semi-global separatrix
has finite transcendence degree over $\mathbb C$.

\begin{lemma}\label{L:D2>0}
    Let $S$ be a smooth complex surface containing a divisor with compact support and positive
    self-intersection. Then there exists an effective divisor $D$ on $S$ with compact support
    and of positive self-intersection  such that  $D\cdot C >0$ for every curve contained in its support.
    Moreover, the transcendence degree over $\C$  of the field of meromorphic functions of $S$ is at most two.
\end{lemma}
\begin{proof}
    Let $\sum m_i C_i$ be a divisor on $S$ with compact support and $(\sum m_i C_i)^2 >0$. If we write this divisor as a difference of
    effective divisors without common irreducible components  $\sum m_i C_i = D_+ - D_-$  then
    it is clear that $D_0= D_+ + D_-$ is an effective  divisor with $D_0^2 >0$.

    Let $D_0 = P + N$ be the Zariski decomposition of $D_0$ in the sense of   \cite[Theorem 3.3]{MR2877664}, i.e.
    both $P$ and $N$ are effective, $P$ is nef,  $P\cdot C =0$ for every $C$ is the support of $N$, and the restriction of
    the intersection form to the support of $N$ is negative definite.  Since $D_0^2>0$ then  $P^2 > 0$.

    Replace $S$ with a sufficiently small neighborhood of a connected component of the support
    of $P$. Thus $P$ is a nef divisor with connected support. The proof of \cite[Lemma 4.9]{MR4150930}
    gives a divisor $D$ with the sought properties and the same support as $P$.

    To conclude, observe that the formal completion of $S$ along $|D|$ satisfies the
    hypothesis of \cite[Theorem 6.7]{MR232780} which in its turn implies the bound on the transcendence degree over $\C$ of
    the field $\mathbb C(S)$ of meromorphic functions of $S$.
\end{proof}

\begin{cor}\label{C:positivesemi-global}
    Let $\F$ be a codimension one foliation with simple singularities on a projective manifold $X$ of
    dimension three. If $S$ is a semi-global separatrix with intersection form having at least one positive
    eigenvalue then the Zariski closure of $S$ is a $\F$-invariant algebraic surface $\overline S$.
\end{cor}
\begin{proof}
    On the one hand, since $X$ is projective, the restrictions of rational functions on $X$ to $S$ define meromorphic functions on
    $S$. Therefore, the transcendence degree of the field of meromorphic functions of $S$ (or rather of its normalization)
    is at least two. On the other hand, Lemma \ref{L:D2>0} implies that $\trdeg \C(S)$ is at most two. Hence the Zariski closure of $S$ has dimension two and the result follows.
\end{proof}

\subsection{Proof of Theorem \ref{THM:Logsep}}
We will make use of the following elementary observation.

\begin{lemma}\label{L:elementar}
    Let $A \in \End(\mathbb Q^n)$, $b \in \mathbb Q^n$, and $x= (x_1, \ldots, x_n) \in \mathbb C^n$ be such that $A(x) = b$.
    If the $\mathbb Z$-submodule of $\mathbb C/\mathbb Q$ generated by the entries $x_i \mod \mathbb Q$ of $x$ has rank $k$ then
    $\dim \ker A \ge k$.
\end{lemma}

Theorem \ref{THM:Logsep} of the introduction is implied by our next result.

\begin{thm}\label{T:Logsep}
    Let $\F$ be a codimension one foliation with simple singularities on a projective manifold $X$ of dimension three.
    Let $S\subset X$ be a  semi-global separatrix for $\F$ and let $\nabla$ be the pull-back of Bott's connection
    to the normalization of $S$. If
    \begin{enumerate}
        \item\label{I:rank2} the classes of the residues of $\nabla$ modulo $\mathbb Q$ generate a $\mathbb Z$-submodule of $\mathbb C/ \mathbb Q$ of rank
        at least two ; or
        \item\label{I:mixedtype} the support of the irregular divisor of $\nabla$ is non-empty and does not contain the
        support of the residue divisor of $\nabla$
    \end{enumerate}
    then the Zariski closure of $S$ is a $\F$-invariant algebraic surface.
\end{thm}
\begin{proof}
    Let $n: S^n \to S$ be the normalization of $S$ and let $\nabla$ be the flat meromorphic connection on $S^n$ induced
    by Bott's connection for $\F$. Let $C_1, \ldots, C_k$ be the irreducible components of $|(\nabla)_{\infty}|$.
    Theorem \ref{T:residue for connections} implies that $c(n^*N^*_{\F}) = - \Res(\nabla) = - \sum \lambda_i C_i$ for complex numbers
    $\lambda_i = \Res_{C_i}(\nabla)$. Assumption (\ref{I:rank2}) on the rank of the subgroup of $\mathbb C/ \mathbb Q$ generated by classes modulo $\mathbb Q$ of the residues of
    $\nabla$ implies, by means of Lemma \ref{L:elementar}, that the intersection matrix of $(\nabla)_{\infty}$ has kernel of dimension at least two. Hence there exists
    two linearly independent divisors $D_1$ and $D_2$ supported on $C_1, \ldots, C_k$ with zero self-intersection. Using the connectedness
    of $n^{-1}(\sing (\F))$ we produce a compact effective divisor $D$ contained in $S^n$ with positive self-intersection. Likewise,
    Assumption (\ref{I:mixedtype}), by means of Proposition \ref{P:irregularconnections}, implies that intersection form of $C_1, \ldots, C_k$ has at least one positive eigenvalue. In both cases, the result follows from Corollary \ref{C:positivesemi-global}.
\end{proof}

\subsection{Proof of Corollary \ref{COR:generic sing}} Under the assumptions of Corollary \ref{COR:generic sing}, the singular point $p$
is a simple singularity for $\F$ as the lemma below shows.

\begin{lemma}\label{L:p is simple}
    Let $\F$  be a germ of foliation on $(\mathbb C^3,0)$ defined by
    \[
        \omega = xyz\left( \alpha \frac{dx}{x} + \beta  \frac{dy}{y} + \gamma \frac{dz}{z} \right) \mod \, \mathfrak m_p^2 \Omega^1_{X,p} \, ,
    \]
    where $\alpha, \beta, \gamma$ are $\mathbb Z$-linearly independent complex numbers.  Then $0$ is a simple singularity for
    $\F$.
\end{lemma}
\begin{proof}
    Consider the unique germ of vector field $v$ on $(\mathbb C^3,0)$ such that $d\omega = i_v dx \wedge dy \wedge dz$.
    Our assumptions on $\omega$ implies that $v$ is a  vector field with semi-simple linear part and eigenvalues $\alpha - \beta$, $\alpha - \gamma$, and $\beta - \gamma$. The $\mathbb Z$-linear independence of $\alpha, \beta, \gamma$ implies that $v$ is non resonant
    and therefore, by Jordan-Chevalley decomposition, $v$ itself is semi-simple and formally linearizatible, see \cite[Proposition 1]{MR647488}. Therefore, in suitable formal coordinates which we still denote by $x,y,z$, $\F$ is defined by
    a (formal) $1$-form $\hat \omega$ such that $d \hat{\omega} = (\alpha - \beta) x dy \wedge dz + (\alpha - \gamma) y dx \wedge dz + (\beta - \gamma) z dx \wedge dy$. Therefore, by integration of $d \hat{\omega}$, we deduce that the foliation $\F$ is defined by  $\hat \omega = xyz\left( \alpha \frac{dx}{x} + \beta  \frac{dy}{y} + \gamma \frac{dz}{z} \right) + df$ for
    some $f \in \mathbb C[[x,y,z]]$. One final change of coordinates of the form $x \mapsto x \cdot u$, where $u \in \mathbb C[[x,y,z]]$ is a unit, proves that $\F$ is formally conjugated to a foliation defined by a logarithmic $1$-form   with simple singularities. The definition of simple singularities for codimension one foliations implies that $0$ is a simple singularity of $\F$.
\end{proof}

Without loss of generality, we may assume that $\F$ has simple singularities since we can replace $\F$ by a reduction of singularities of it, with no center over $p$ (Lemma \ref{L:p is simple}).

According to \cite[Subsection 5.3.1]{MR1760842}, there are exactly three distinct germs of separatrices through $p$. Proposition \ref{P:extend sep} implies that we can extend these three germs to three, not necessarily distinct, semi-global separatrices for $\F$. Example \ref{E:generic sing} shows that the $\mathbb Z$-submodule of $\mathbb C$ generated by the  residues of $\nabla^B$ along the semi-global separatrix $S_x$ tangent to $\{ x=0\}$ contains the complex numbers $1 - \beta/\alpha$ and $1 - \gamma/\alpha$. Since $\alpha,\beta, \gamma$ are $\mathbb Z$-linearly independent
by assumption, we can apply  Theorem \ref{T:Logsep} to conclude that the Zariski closure of $S_x$ is an algebraic surface. The same argument shows the algebraicity of the other two separatrices through $p$. \qed

\subsection{Negative semi-definite intersection form and Ueda theory} We now turn our attention to
semi-global separatrices with  negative semi-definite (but not negative definite) intersection forms, \ie the eigenvalues of the intersection
form are all non-positive and one of them is strictly negative.

In this situation, one cannot expect the algebraicity of the semi-global separatrix. To have a concrete example, start with a foliation
$\G$ on a projective surface $Y$ and consider the foliation $\F$ on $X = Y \times C$ equal to $\pi^*\G$ for $\pi : X \to Y$ the natural projection.
If $\G$ is singular then the pull-back of any germ of separatrix for $\G$ is a semi-global separatrix for $\F$.

Ueda  studies in \cite{Ueda}  neighborhoods of smooth compact curves with topologically trivial normal bundles contained on (not necessarily compact) complex surfaces. See \cite{Neeman} for an insightful exposition and \cite[Section 2]{MR3940902} for applications in foliation theory.

\begin{thm}\label{T:Ueda}
    Let $C \subset X$ be a smooth curve contained in a germ of smooth surface $X$. If $C^2=0$  then
    then exactly one of the following possibilities holds true.
    \begin{enumerate}
        \item\label{I:Ueda 1} There exists a non-constant morphism $f:(X,C) \to (\mathbb C,0)$ with central fiber supported on $C$; or
        \item\label{I:Ueda 2} there exists a formal closed logarithmic $1$-form $\omega$ with $\Res(\omega) = C$ and period group dense in $(i\mathbb R,+)$ ; or
        \item\label{I:Ueda 3} there exists a $C^{\infty}$ strictly plurisubharmonic function $f: X - C \to \mathbb R$  that tends to $\infty$ when approaching $C$. Consequently, $C$ is the intersection of nested pseudoconcave neighborhoods.
    \end{enumerate}
\end{thm}

Another way of stating Item (\ref{I:Ueda 2}), which also clarifies the meaning of the period group of a formal logarithmic $1$-form $\omega$, is as follows. There exists a covering $\mathcal U= \{ U_i \}$ of $X$, formal functions $f_i$ on the formal the completion of $C\cap U_i$ on $U_i$, and constant functions $\lambda_{ij} \in S^1$ defining a cocycle $ \lambda \in Z^1(\mathcal U , S^1)$ such that $f_i = \lambda_{ij} f_j$ and the representation   $\rho : \pi_1(C) \to S^1$ induced by $\lambda$ has dense image.

A curve that satisfies Item (\ref{I:Ueda 3}) behaves in many ways like a curve with positive self-intersection as
the following result by  Andreotti \cite[Theorem 4]{MR152674} shows.
\begin{thm}\label{T:Andreotti}
    Let $X$ be a complex manifold and $U\subset X$ be a relatively compact open subset. If $U$ is strictly pseudoconcave then
    the field of meromorphic functions of $X$ has transcendence degree bounded by the dimension of $X$.
\end{thm}
It is interesting to observe that the field of formal meromorphic functions of the formal completion of $X$ along a curve satisfying the assumptions of Theorem \ref{T:Ueda} has infinite transcendence degree over $\mathbb C$, see \cite[proof of Proposition 5.1.1]{Hironaka-Matsumura}. To wit, the field of (convergent) meromorphic functions has
finite transcendence degree while the field of formal meromorphic functions has infinite transcendence degree.

\begin{remark}\label{R:Ueda 1 or 2}
    Item (\ref{I:Ueda 1}) can only happen when the normal bundle of $C$ corresponds to a torsion point of $\Pic^0(C)$, Item (\ref{I:Ueda 2}) can only happen when the normal bundle of $C$ is not torsion, while Item (\ref{I:Ueda 3}) can happen in both situations. Ueda also proves a more precise version of Item (\ref{I:Ueda 2}) of Theorem \ref{T:Ueda}. If the normal bundle of $C$ satisfies certain additional assumptions of Diophantine nature, the formal closed logarithmic $1$-form $\omega$ is actually convergent.
\end{remark}

Theorem \ref{T:Ueda} implies the following result about semi-global separatrices with semi-definite intersection form.

\begin{thm}\label{T:0}
    Let $\F$ be a codimension one foliation on a projective manifold $X$ of
    dimension three. If $S$ is a semi-global separatrix with such that the pull-back of $\sing(\F)\cap S$ to the normalization
    of $S$ is a smooth curve with torsion normal bundle then at least one of the following assertions holds true
    \begin{enumerate}
        \item the foliation $\F$ is the pull-back, under a rational map, of a foliation $\G$ on a projective surface; or
        \item the Zariski closure of $S$ is an algebraic surface invariant by $\F$.
    \end{enumerate}
\end{thm}
\begin{proof}
    Let $n: S^n \to S$ be the normalization of $S$ and let $C = n^{-1}(\sing(\F))$. By assumption $C$ is a smooth curve.
    Assume that the Zariski closure of $S$ is not algebraic. Since $X$ is algebraic, the transcendence degree $d$ of the field of meromorphic functions of $S^n$ is at least the dimension of $S$. Since we are assuming that $S$ is not algebraic, we have that $d$ is strictly bigger than two. Therefore, combining  Theorem \ref{T:Ueda}, Theorem \ref{T:Andreotti}, and Remark \ref{R:Ueda 1 or 2}, we obtain the  existence of a proper fibration  $f : (S^n,C) \to (\mathbb C,0)$ having $C$ equal to the support of the fiber over $0$. The (local) fibration $f$ defines a germ of curve on the Hilbert scheme of $X$ whose points correspond to curves on $X$ tangent to $\F$. Since tangency to a foliation is a closed condition \cite[Proposition 2.1]{MR2248154}, the points of the Zariski closure of this germ of curve on the Hilbert scheme  correspond to one-dimensional subschemes that are still tangent to $\F$. Since $S$ is not algebraic by assumption, this family of subschemes must cover $X$. We apply   \cite[Lemma 2.4]{MR3842065} to conclude that $\F$ is the rational pull-back of a foliation $\G$ on a projective surface.
\end{proof}

Unfortunately, the assumptions of Theorem \ref{T:0} are strong and certainly not optimal. Still assuming  the smoothness $n^{-1}(\sing(\F)$, it would be interesting to know if one can replace torsion normal bundle with numerically trivial normal bundle (\ie degree zero). In all the examples we are aware of, when the normal bundle has degree zero but is not torsion, the Zariski closure of the semi-global separatrix is a projective surface. Also, the smoothness of $n^{-1}(\sing(\F))$ along a semi-global separatrix is a rather unnatural assumption. It is essential for the proof presented above, as it relies on Theorem \ref{T:Ueda} which is not available for non-smooth curves. It is unclear whether Ueda's result holds or not for arbitrary (non-reduced) curves with singularities. A partial result pointing toward a positive answer to this question was obtained by  Koike. He proved a version of Ueda's result for reduced simple normal crossing curves, with contractible dual graph, trivial normal bundle, and non-zero Ueda class in every irreducible component of the curve,  see \cite[Theorem 1.6]{Koike}. In the next section, we will prove another version of Ueda's result valid for the irregular divisor of closed meromorphic $1$-forms on compact Kähler surfaces. Although this result will not be useful to the study of semi-global separatrices, it provides some hope of establishing a version of Theorem \ref{T:Ueda} that would apply to any semi-global separatrix with negative semi-definite intersection form.

\section{Polar divisor of meromorphic 1-forms of the second kind}\label{S:Ueda}

\subsection{Ueda theory for the polar divisor of $1$-forms of second kind}
Ueda's arguments used to prove Theorem \ref{T:Ueda}, acquire a particularly simple form when the curve is contained in a compact Kähler surface
and is the polar divisor of a closed meromorphic $1$-form of the second kind. Moreover, under these strong assumptions,
the smoothness of the curve is no longer relevant.

\begin{thm}[Theorem \ref{THM:Ueda} of Introduction]\label{T:Ueda projective}
    Let $X$ be a compact Kähler surface and let $I \in \Div(X)$ be an effective divisor with connected support satisfying
    \[
        \mathcal O_X(I) \otimes \frac{\mathcal O_X}{\mathcal O_X(-I)} \simeq \frac{\mathcal O_X}{\mathcal O_X(-I)} \, .
    \]
    Then exactly one of the following assertions holds true
    \begin{enumerate}
        \item there exists a  non-constant morphism $f:X \to C$ onto a curve with connected fibers and mapping $|I|$ to a point, or
        \item there exists an exhaustion of $X-|I|$ that is strictly plurisubharmonic at a neighborhood of $|I|$. Consequently, $X-|I|$ is holomorphically convex  and, after the contraction of finitely many compact curves, becomes a (singular) Stein surface.
    \end{enumerate}
\end{thm}
\begin{proof}
    Let $\omega$ be a closed meromorphic $1$-form of the second kind and without base points on $X$ given by Theorem  \ref{T:realizaII}, \ie,  $(\omega)_{\infty} = I + I_{\text{red}}$. After adding a suitable closed holomorphic $1$-form to $\omega$ we can (and will) assume, without loss of generality, that the class of $\omega$ in $H^1(X,\mathbb C)$ belongs $\overline{H^0(X,\Omega^1_X)}$. If the class of $\omega$ in $H^1(X,\mathbb C)$ is zero, then $\omega$ is equal to $dh$ for some meromorphic function $h \in \mathbb C(X)$. Since $\omega$ has no base points, $h$ defines a morphism to $\mathbb P^1$ whose Stein factorization is the sought morphism $f:X \to C$.

    From now on, assume that the class of $\omega$ is non-zero in $H^1(X,\mathbb C)$ and let $\alpha \in H^0(X,\Omega^1_X)$ be the unique holomorphic $1$-form such that $[\omega] = [\overline \alpha]$ in $H^1(X,\mathbb C)$. By assumption $\alpha \neq 0$. Moreover $\omega \wedge \alpha \neq 0$, since $[\omega] \wedge [\alpha]$ in $H^2(X,\C)$ coincides with $[\overline{\alpha}] \wedge [\alpha] \neq 0$ in the same group.  We observe that this implies the existence of an irreducible curve $C$ contained in the support
    of $I$ such that the pull-back of $\alpha$ to $C$ does not vanish identically. Indeed, if the pull-back of $\alpha$ to every single irreducible component of $|I|$ is identically zero, then we can construct, unambiguously, a primitive $h:U \to \C$ for $\alpha$ defined on an open set $U\subset X$ containing $|I|$ by imposing that $\restr{h}{|I|} =0$. Therefore $\restr{\alpha}{U}$ is exact, and the same holds true for $\restr{\omega}{U}$. Hence, the foliations on $U$ defined by $\restr{\alpha}{U}$ and $\restr{\omega}{U}$ are both     fibrations ($\omega$ has no base points) having a fiber supported on $I$. Let $L \subset U$ be any compact leaf of the foliation defined by $\restr{\omega}{U}$. Its image under $h$ (the primitive of $\alpha$) must be a point. It follows that $L$ is also a leaf of the foliation defined by $\alpha$. This leads to the contradiction $\omega\wedge \alpha =0$, showing the existence of a curve $C \subset |I|$ such that  $\alpha$ is not identically zero after pulled back to $C$.

    We will construct an exhaustion of $X -|I|$ that is strictly plurisubharmonic at a neighborhood
    of $|I|$. For that, start by considering the function
    \begin{align*}
        F : X- |I| &\longrightarrow \mathbb C \\
        x &\longmapsto \left\vert \int_{x_0}^x \omega -\overline{\alpha} \right\vert^2 \, .
    \end{align*}
    Clearly, $F$ is an exhaustion of $X-|I|$ since $\alpha$ is holomorphic and $\omega$ has poles along $I$.
    The Levi form of $F$ (complex Hessian) is
    \[
        \partial \overline \partial F =  \omega \wedge \overline{  \omega} +  \alpha \wedge \overline{   \alpha}
    \]
    showing that $F$ is a plurisubharmonic function. Moreover,
    \[
        \partial \overline \partial F \wedge \partial \overline \partial F = 2  \omega \wedge \overline{  \omega} \wedge   \alpha \wedge \overline{ \alpha} \, ,
    \]
    showing that $F$ is strictly plurisubharmonic outside the zero set of $\omega \wedge \alpha$. Since both $\omega$ and $\alpha$ are holomorphic $1$-forms on $X-|I|$, the closure of the zero set of $\omega\wedge \alpha$ is a divisor $S$. If the support of $S$ has no irreducible component intersecting $|I|$ but not contained in $|I|$ then  $F$ is a  exhaustion of $X-|I|$ which is strictly plurisubharmonic at a neighborhood of infinity. This shows that $X-|I|$ is a union of relatively compact strictly pseudoconvex open subsets. A result by Grauert \cite[Theorem 1]{MR98847} implies that $X-|I|$ is holomorphically convex.  When the support of $S$ has an irreducible component which intersects $|I|$ but is
    not contained in it, we will show how to alter $F$ in order to obtain an exhaustion with the same properties as above.

    Let $C$ be an irreducible component of the support of $I$ such that $\alpha$ is not identically zero when pulled back to $C$. Observe that this implies that $C$ is not contained in  $|S|$ (since $\omega \wedge \alpha$ has a pole along $C$) and  that $\partial \overline \partial F \wedge \partial \overline \partial F$ is arbitrarily large at neighborhoods of points in $|C| - |S|$ thanks to the poles of $\omega$.

    The intersection matrix of the divisor $I_{\text{red}} - C$ is negative definite. Theorem \ref{thm:Grauert} (Grauert's contraction theorem)  implies the existence of a bimeromorphic morphism $\pi: X \to Y$ that contracts $|I_{\text{red}} - C|$ to  finitely many points (as many as the connected components of $|I_{\text{red}} - C|$) and is an isomorphism elsewhere. The image of the divisor $S$ intersects
    $\pi(C)$ at finitely many points. Let $y \in Y$ be one of those points. Choose a small neighborhood $U$ of $y$, a strictly plurisubharmonic function on $g : U \to \mathbb R$, and a compactly supported bump  function $\rho : U \to [0,1]$ identically equal to one at a neighborhood of $y$. If $\varepsilon >0$ is small enough then the function $F + \varepsilon \rho g$ is strictly plurisubharmonic at
    \[
        \left(Y -( |\pi(C)| \cup |\pi(S)|) \right) \cup (U - \pi(C)) \simeq \left( X - (|I| \cup  |S|) \right) \cup ( \pi^{-1}(U) - |I|) ,
    \]
    since $\partial \overline \partial F \wedge \partial \overline \partial F$ is arbitrarily large at small neighborhoods of points
    in $\pi(C) - \pi(S)$. Proceeding in this way, we can alter the plurisubharmonic exhaustion $F$ to obtain  an exhaustion that is strictly plurisubharmonic at a neighborhood  of infinity. It follows, as before, that  $X-|I|$ is holomorphically convex.
\end{proof}

\subsection{Formal principle for curves with trivial normal bundle on projective surfaces}\label{SS:formal}
The two $1$-forms, $\omega$ and $\alpha$, appearing in the proof of Theorem \ref{T:Ueda projective} are key
to the investigation of Grauert's formal principle for curves on projective surfaces carried out in \cite{MR4333432}.

\begin{thm}\label{T:formal}
    Let $(S,C)$ be a pair where $S$ is a smooth projective surface and $C$ is a smooth curve contained in $S$.
    Assume that the normal bundle of $C$ in $S$ is trivial and $C$ is not a fiber of a fibration in $S$.
    If $(S',C')$ is another pair with $S'$ projective and such that the formal completion of $C$ in $S$ is
    formally isomorphic to the formal completion of $C'$ in $S'$ then there exists a birational map between $S$ and
    $S'$ that sends, biregularly, a neighborhood of $C$ in $S$ to a neighborhood of $C'$ in $S'$.
\end{thm}

The argument used to establish Theorem \ref{T:formal} goes as follows.
Since $C$ has trivial normal bundle, Theorem \ref{T:realizaII} implies the existence of a closed meromorphic $1$-form
$\omega$ of the second kind having polar divisor equal to $2C$. Moreover, after adding a suitable holomorphic $1$-form, we
can assume that the periods of $\omega$ are anti-holomorphic and non-zero because $C$ is not a fiber of a fibration by assumption.
If $\alpha$ is a holomorphic $1$-form with the same periods as $\overline \omega$, then
the vector space of closed rational $1$-forms generated by $\omega$ and $\alpha$ is canonically
associated with the pair $(S,C)$ and must be preserved by formal equivalences. With some extra work, and using the geometric structure defined by $\omega$ and $\alpha$, one can show the convergence of the formal equivalence. Finally, the existence of a birational map inducing the formal equivalence follows from extension properties of meromorphic functions defined on a neighborhood of $C$. Details can be found in \cite{MR4333432}.

\section{A remark on Stein complements}\label{S:Stein}
As already mentioned in the Introduction, Serre showed the existence  of a $\mathbb P^1$-bundle over an elliptic curve admitting a section with trivial normal bundle and having Stein complement \cite[Chapter 6, Example 3.2]{MR0282977}.
It follows from Ueda's Theorem \ref{T:Ueda} the existence of curves $C$ of genus $g \ge 2$ on projective surfaces having Stein complements. It suffices to consider suspensions of non-abelian representations $\rho: \pi_1(C) \to \Aff(\mathbb C)$ with prescribed unitary linear part  to obtain $\mathbb P^1$-bundles over $C$ with a section with normal bundle determined by the unitary linear part of $\rho$ and Stein complement.

Our final result shows that the complements of hypersurfaces with numerically trivial normal bundles on compact Kähler manifolds of dimension at least three are never Stein.

\begin{thm}[Theorem \ref{THM:Stein} of Introduction]\label{T:Steindim2}
    Let $X$ be a compact Kähler manifold and $Y \subset X$ be a smooth hypersurface with numerically trivial normal bundle.
    If $X-Y$ is Stein then $\dim X=2$.
\end{thm}

\subsection{Existence of a closed meromorphic $1$-form with coefficients on a flat line-bundle}
To prove Theorem \ref{T:Steindim2}, we will assume that $\dim X \ge 3$ and will look for a contradiction.
We start by proving the existence of a closed meromorphic $1$-form with coefficients on a unitary flat line-bundle
and poles on the hypersurface under study when the ambient spaces have  dimensions at least three.

\begin{lemma}\label{L:omegadim3}
    Let $X$ be a compact Kähler manifold and $Y \subset X$ be a smooth hypersurface with numerically trivial normal bundle.
    If $\dim X\ge 3$ and $X-Y$ is Stein then there exists
    \begin{enumerate}
        \item a line-bundle $\mathcal L$ on $X$ with zero Chern class such that $\restr{\mathcal L}{Y}=\mathcal O_X(Y)$;
        \item a flat unitary connection $\nabla$ on $\mathcal L$; and
        \item\label{I:xxx3} a $\nabla$-closed meromorphic $1$-form $\omega$ with coefficients in $\mathcal L$, the rank one local system of flat sections of  $(\mathcal L,\nabla)$,  such that  $(\omega)_{\infty} = 2Y$ and   the class of $\omega$ in $H^1(X, \mathbb L)$ is non-zero and lies in $\overline{H^0(X,\Omega^1_X \otimes {\mathcal L}^*)}$.
    \end{enumerate}
    Moreover, the $1$-form $\omega$ is unique up to multiplication by non-zero constants.
\end{lemma}
\begin{proof}
    Since $Y$ is Kähler and $N_Y$ is numerically trivial, there exists a unique flat unitary connection on $N_Y$.
    Let $\rho_Y : \pi_1(Y) \to S^1$ be its monodromy and  $\iota:Y\to X$ be the natural inclusion.
    The proof of Lefschetz theorem for homotopy groups of the complement of ample divisors presented in \cite[Theorem 3.1.21]{MR2095471} works equally well to prove $\iota_*:\pi_i(Y) \to \pi_i(X)$ is an isomorphism that for $i \le \dim X-2$ whenever $X-Y$ is Stein. In particular,
    there exists a representation $\rho_X:\pi_1(X) \to S^1$ such that $\rho_X \circ \iota_* = \rho_Y$. The unique flat unitary line-bundle $(\mathcal L,\nabla)$ on $X$ with monodromy $\rho_X$ satisfies $\restr{\mathcal L}{Y} = N_Y$.

    The existence of the $\nabla$-closed meromorphic $1$-form $\omega$ with the properties listed in Item (\ref{I:xxx3})  follows from \cite[Theorem B]{MR3940902}, see also \cite[Theorem B]{MR4333432}, as previously indicated in Remark \ref{R:thmB}.
\end{proof}

\subsection{Leaves of foliations on Stein manifolds}
The following lemma is well-known to experts. The proof we present here is adapted from \cite[proof of Theorem 6.4]{MR924673}.

\begin{lemma}\label{L:Stein}
    Let $\F$ be a holomorphic foliation of positive dimension (\ie not a foliation by points) on a Stein manifold $X$. If $L$ is a leaf of $\F$ then $L$ intersects the complement of every compact subset of $X$.
\end{lemma}
\begin{proof}
    Embed $X$ as a closed analytic subset of $\mathbb C^n$. Let $p\in L$ be an arbitrary point. Since $T_{\F}$ is coherent, it is generated
    by global sections according to Cartan's Theorem A. Hence there exists a vector field $v \in H^0(X,T_{\F})$ that does not vanish at $p$.
    Using Cartan's Theorem B, we can extend $v$ to a holomorphic vector field on $\mathbb C^n$.

    Aiming at a contradiction, let us assume that
    the orbit of $v$ through $p$ is contained in a compact subset $K\subset X$. Let $\varphi : \mathbb D(0,r) \to \mathbb C^n$ be an orbit of $v$ centered at $p$, \ie     $\varphi'(x) = v_{\varphi(x)}$ and $\varphi(0)=p$, with maximal possible $r$.  The theorem of existence and uniqueness of solutions of holomorphic differential equations implies the existence of a real number $\delta>0$ (depending only on the norm of $\restr{v}{K'}$, where $K' \subset \C^n$ is any compact set with interior containing $K$)  such that for every point $x \in K$, there exists an orbit of $v$ centered at $x$ defined on the disc $\mathbb D(0,\delta)$.
    Using this, we see that $r$ must be equal to infinity. But if $r = \infty$, then Liouville's theorem implies that $\varphi$ is constant contradicting the non-vanishing of $v$ at $p$.
\end{proof}

\subsection{Proof of Theorem \ref{T:Steindim2}/Theorem \ref{THM:Stein}}
Aiming at a contradiction, assume that $\dim X \ge 3$. Let $\omega$ be $\nabla$-closed meromorphic $1$-form given by
Lemma \ref{L:omegadim3}. If we choose an open covering $\mathcal U=\{U_i\}$ of $X$ and a trivilization
of $\mathcal L$ over $\mathcal U$ with transition functions $\lambda_{ij}$ equal to constant functions of modulus one, then
$\omega$ is represented by a collection of closed meromorphic $1$-forms $\{\omega_i \in \Omega^1_X(2Y)(U_i)\}$
such that $\omega_i = \lambda_{ij} \omega_j$.

Let $\alpha \in H^0(X,\Omega^1_X \otimes \mathcal L^*)$ be the $1$-form with coefficients in $\mathcal L^*$ representing
the same class as $\omega$ in $H^1(X,\mathbb L)$. Hence $\alpha$ may be represented in the covering $\mathcal U$ above
by a collection of closed holomorphic $1$-forms $\{ \alpha_i \in \Omega^1_X(U_i) \}$ such that $\alpha_i = \overline{\lambda_{ij}} \alpha_j$
since complex conjugation coincides with the inverse for complex number of modulus one.

Because the class of $\omega - \overline{\alpha}$ in $H^1(X,\mathbb L)$ is zero, we can choose {\it meromorphic} pluriharmonic functions $f_i$ such that the holomorphic parts have simple poles along $Y\cap U_i$,  the anti-holomorphic parts have no poles,  $df_i = \omega_i - \overline{\alpha}$ and $f_i = \lambda_{ij} f_j$. Let $F : X - Y \to [0,\infty)$ be
the $C^{\infty}$-function defined over $U_i$ as $F = |f_i|^2$. Observe that $F$ is  an exhaustion of $X-Y$ and thus its  level sets  are compact.

Consider the codimension two holomorphic foliation $\mathcal G$ defined on $X$ by the twisted $1$-forms $\omega$ and $\alpha$. The leaves of $\restr{\mathcal G}{X-Y}$ are contained in  level sets of $F$. But this is impossible, since the leaves of holomorphic foliations on
Stein manifolds scape every compact subset according to Lemma \ref{L:Stein}. \qed

\providecommand{\bysame}{\leavevmode\hbox to3em{\hrulefill}\thinspace}
\providecommand{\MR}{\relax\ifhmode\unskip\space\fi MR }
\providecommand{\MRhref}[2]{%
  \href{http://www.ams.org/mathscinet-getitem?mr=#1}{#2}
}
\providecommand{\href}[2]{#2}

\end{document}